\documentclass[a4paper,11pt]{article}
\usepackage{amsmath,amssymb,amsthm,stmaryrd} 
\usepackage{mathtools}
\usepackage{faktor}
\usepackage{float} 
\usepackage{color}
\usepackage{enumerate}
\usepackage[dvipsnames]{xcolor}
\usepackage[normalem]{ulem}   

\usepackage{tikz}
\usetikzlibrary{patterns}
\usetikzlibrary{arrows.meta}

\numberwithin{equation}{section}

\theoremstyle{theorem}
\newtheorem{theorem}{Theorem}[section]
\newtheorem{lemma}[theorem]{Lemma}

\newtheorem{proposition}[theorem]{Proposition}
\newtheorem{corollary}[theorem]{Corollary}
\theoremstyle{definition}
\newtheorem{definition}[theorem]{Definition}
\newtheorem{remark}[theorem]{Remark}

\renewcommand{\a}{\alpha}
\newcommand{\ova}{{\overline\alpha}}
\renewcommand{\b}{\beta}
\newcommand{\g}{\gamma}
\renewcommand{\d}{\delta}

\newcommand{\x}{\xi}

\newcommand{\sub}{\subset}

\newcommand{\R}{\mathbb{R}}
\newcommand{\N}{\mathbb{N}}

\newcommand{\Z}{\mathbb{Z}}

\newcommand{\X}{\mathbb{X}}

\newcommand{\MM}{\mathcal{M}}
\newcommand{\NN}{\mathcal{N}}
\newcommand{\FF}{\mathcal{F}}
\newcommand{\PP}{\mathcal{P}}
\newcommand{\TP}{\mathcal{TP}}
\newcommand{\CP}{\mathcal{CP}}
\newcommand{\we}{\wedge}

\newcommand{\hel} {
\hskip2.5pt{\vrule height7pt width.5pt depth0pt}
\hskip-.2pt\vbox{\hrule height.5pt width7pt depth0pt}
\,}
\newcommand{\restr}{\hel}
\newcommand{\srestr}{
\hskip2.5pt{\vrule height7pt width.5pt depth0pt}
\hskip-.2pt\vbox{\hrule height.5pt width7pt depth0pt}
\!\!\raisebox{1pt}{{$^*$}}\,}
\newcommand{\pf} {\,\!_\#\,} 

\renewcommand{\j}{\jmath}
\renewcommand{\i}{\imath}

\newcommand{\sm}{\backslash}
\newcommand{\pt}{\partial}

\newcommand{\eps}{\varepsilon}

\newcommand{\oo}{\infty}
\newcommand{\loc}{_{\textnormal{loc}}}

\newcommand{\ov}{\overline}
\newcommand{\wh}{\widehat}
\newcommand{\wt}{\widetilde}
\newcommand{\h}{\mathcal{H}}
\newcommand{\dw}{\downarrow}
\newcommand{\up}{\uparrow}
\newcommand{\cd}{\cdot}
\newcommand{\longto}{\longrightarrow}
\newcommand{\Id}{\text{Id}}
\renewcommand{\t}{\times}

\newcommand{\om}{\omega}
\newcommand{\ds}{\displaystyle}

\newcommand{\lb}{\llbracket}
\newcommand{\rb}{\rrbracket}
\newcommand{\la}{\langle}
\newcommand{\ra}{\rangle}
\newcommand{\be}{\begin{equation}}
\newcommand{\ee}{\end{equation}}

\newcommand{\XXint}[3]{{\setbox0=\hbox{$#1{#2#3}{\int}$}
      \vcenter{\hbox{$#2#3$}}\kern-.5\we0}}

\newcommand{\void}{\varnothing}

\newcommand{\st}{\stackrel}
\newcommand{\lt}{\left}
\newcommand{\rt}{\right}

\newcommand{\M}{\mathbb{M}}
\newcommand{\F}{\mathbb{F}}
\newcommand{\Mt}{\mathbb{M}^{\text{\tiny$\times$}}}
\newcommand{\Nt}{\mathbb{N}^{\text{\tiny$\times$}}}

\newcommand{\Fw}{\mathbb{F}^{\text{\tiny$\we$}}}
\newcommand{\Mw}{\mathbb{M}^{\text{\tiny$\we$}}}
\newcommand{\Nw}{\mathbb{N}^{\text{\tiny$\we$}}}
\newcommand{\chiw}{\chi^{\text{\tiny$\we$}}}

\newcommand{\GG}{\mathcal{G}}
\newcommand{\CC}{\mathcal{C}}

\DeclareMathOperator{\OO}{O}

\DeclareMathOperator{\supp}{supp}

\DeclareMathOperator{\Span}{span}

\DeclareMathOperator{\Sl}{Sl}

\DeclareMathOperator{\Gr}{Gr}

\newcommand{\red}{\color{red}}

\newcommand{\comfig}[1]{}

\title{A deformation theorem for tensor flat chains and applications (complement to ``Tensor rectifiable $G$-flat chains'')}
\author{M. Goldman\footnote{ CMAP, CNRS, \'Ecole polytechnique, Institut Polytechnique de Paris, 91120 Palaiseau,
France, email: michael.goldman@cnrs.fr} \and B. Merlet\footnote{Univ. Lille, CNRS, UMR 8524, Inria- Laboratoire Paul Painlevé, F-59000 Lille, email: benoit.merlet@univ-lille.fr}}

\begin{document}
\maketitle

\begin{abstract}
In this note we extend White's deformation theorem for $G$-flat chains to the setting of $G$-flat \emph{tensor} chains. As a corollary we obtain that the groups of \emph{normal tensor} chains identify with some subgroups of normal chains. Moreover the corresponding natural group isomorphisms are isometric with respect to norms based on the~\emph{coordinate slicing mass}. 

The coordinate slicing mass of a $k$-chain is the integral of the mass of its $0$-slices along all coordinate-planes of codimension $k$. 
The fact that this quantity is equivalent to the usual mass is not straightforward. To prove it, we use the deformation theorem and a partial extension of the restriction operator defined for all chains (not only of finite mass).

On the contrary, except in some limit or degenerate cases, the whole groups of tensor chains and of finite mass tensor chains do not identify naturally with subgroups of chains.

\end{abstract}


\appendix 

This article is an appendix to~\cite{GM_tfc} and is not meant to be read independently. In that paper, the groups $\FF^G_{k_1,k_2}(\R^n)$ of~\emph{tensor} chains in $\R^n=\R^{n_1}\t\R^{n_2}$  with coefficients in an Abelian normed group $G$ are introduced. We recall that they  can be defined as the completion of the tensor product $\FF^\Z_{k_1}(\R^n)\otimes\FF^G_{k_2}(\R^n)$ with respect to its projective norm. \\
We use the same notation as  in the main article. Let us recall that $n=n_1+n_2$ and $k=k_1+k_2$ with $0\le k_l\le n_l$ for $l=1,2$. Moreover, $(e_1,\dots,e_n)$ denotes the standard basis of $\R^n$, $I^n_k$ is the set of subsets of $\{1,\dots,n\}$ with $k$-elements, $\ov\g$ is the complement in $\{1,\dots,n\}$ of $\g\in I^n_k$ and $\X_\g:=\Span\{e_i:i\in\g\}$, more generally we denote $\X_\g(x):=x+\X_\g$ for $x\in\X_{\ov\g}$,  eventually we set $\a:=(1,\dots,n_1)$ so that $\ova=(n_1+1,\dots,n)$. 

\section{Introduction and main results}
The original deformation theorem of Federer and Fleming~\cite{FedFlem60} for normal currents was first improved by Simon~\cite{Simon1983} and later generalized to $G$-flat chains by White~\cite{White1999-1}. Here, we extend it in Theorem~\ref{thm_defthm} to the setting of tensor $G$-flat chains of~\cite{GM_tfc}. The idea is still to ``deform'' a given $(k_1,k_2)$-chain on the $k$-skeleton of some regular grid. This leads to family of projections of this chain on a subgroup of the $(k_1,k_2)$-polyhedral chains with some control on the norms of the projection and of the chain swept during the deformation.

With this result at hand, we characterize the preimage of the group of normal $(k_1,k_2)$-chains by the group morphism $\j_{k_1,k_2}$ introduced in~\cite[Subsection~4.6]{GM_tfc}. Another outcome of this study is the equivalence of the mass and of the slicing mass (both for tensor and classical flat chains). To the best of our knowledge, and while this result is  very natural,  it does seem to appear in the available literature.\medskip

\subsubsection*{Group isomorphisms}
As in the main paper,  we denote
\[
D_k:=\{(k'_1,k'_2): 0\le k'_1\le n_1,\ 0\le k'_2\le n_2,\ k'_1+k'_2=k\}.
\]
The continuous group morphism $\j$ associates to a $k$-chain $A$ a list of $(k'_1,k'_2)$-chains $\j_{k'_1,k'_2}A$ for $(k'_1,k'_2)\in D_k$.  Roughly speaking, the tensor chain $\j_{k_1,k_2}A$ corresponds to the component of $A$ which is of dimension $k_1$ in $\X_\a$ and of dimension $k_2$ in $\X_\ova$. Given $(k_1,k_2)\in D_k$, we are interested in characterizing the $(k_1,k_2)$-chains which identify naturally with usual flat chains. In light of the description of $\j$ these later lie in the following group (already introduced in~\cite[Definition~5.9]{GM_tfc}).
\[
\FF^G_{k,(k_1,k_2)}(\R^n):=\lt\{ A\in \FF^G_k(\R^n): \j_{k'_1,k'_2}A=0 \text{ for } (k'_1,k'_2)\in D_k\sm\{(k_1,k_2)\}\rt\}.
\]
Let us also recall the definition of the subgroups:
\[
\MM^G_{k,(k_1,k_2)}(\X_\b):=\FF^G_{k,(k_1,k_2)}(\X_\b)\cap\MM^G_k(\X_\b),\quad\qquad \NN^G_{k,(k_1,k_2)}(\X_\b):=\FF^G_{k,(k_1,k_2)}(\X_\b)\cap\NN^G_k(\X_\b).
\]
By~\cite[Proposition~5.14]{GM_tfc}, we have the one-to-one Lipschitz continuous group morphisms,
\begin{eqnarray}
\label{j1}
\j_{k_1,k_2}:&\FF^G_{k,(k_1,k_2)}(\R^n)&\!\hookrightarrow\ \FF^G_{k_1,k_2}(\R^n),\\
\label{j2}
\j_{k_1,k_2}:&\MM^G_{k,(k_1,k_2)}(\R^n)&\!\hookrightarrow\ \MM^G_{k_1,k_2}(\R^n),\\
\label{j3}
\j_{k_1,k_2}:&\NN^G_{k,(k_1,k_2)}(\R^n)&\!\hookrightarrow\ \NN^G_{k_1,k_2}(\R^n),
\end{eqnarray} 
where the group $\NN^G_{k_1,k_2}(\R^n)$ of normal $(k_1,k_2)$-chains is formed by the elements $A'\in\FF^G_{k_1,k_2}(\R^n)$ such that $\Nw(A'):=\Mw(A')+\Mw(\pt_1A')+\Mw(\pt_2A')<\oo$. 
The main question we wish to address is whether the above mappings are onto. For the last one and for $k=0$, 
we have established in~\cite[Theorem~5.15]{GM_tfc} that, indeed, $\j_{0,k}$ maps $\NN^G_{k,(0,k)}(\R^n)$ onto $\NN^G_{0,k}(\R^n)$. Using the deformation theorem, see Theorem~\ref{thm_defthm} below, we extend here this result to arbitrary dimensions $k_1,k_2$.
\begin{theorem}\label{thm_Ups_2}~\\
The mapping $\j_{k_1,k_2}$ defines a group isomorphism from $\NN^G_{k,(k_1,k_2)}(\R^n)$ onto $\NN^G_{k_1,k_2}(\R^n)$. 
\end{theorem}
As $\j_{k_1,k_2}$ commutes with restrictions we get the following immediate corollary (recall Definition~6.1 in~\cite{GM_tfc}: $A'\in\MM^G_{k_1,k_2}(\R^n)$ is $(k_1,k_2)$-rectifiable if $A'=A'\restr (\Sigma_1\times \Sigma_2)$ for a $k_l$-rectifiable sets $\Sigma_l\sub\R^{n_l}$). 
\begin{corollary}\label{coro_Ups_4}
The operator $\j_{k_1,k_2}$ maps the group of normal $(k_1,k_2)$-rectifiable chains onto the group of normal rectifiable $(k_1,k_2)$-chains.
\end{corollary}
This result shows that the group of normal rectifiable $(k_1,k_2)$-chains is the right setting for the main result of~\cite{GM_tfc}.\medskip

Before addressing the general  case, let us list what is already known  about the  mappings $\j_{k_1,k_2}$  for extreme values of $k$, $n_1$ and $n_2$.\medskip

\noindent
($*$) If $k=0$ ($k_1=k_2=0$), the boundary of a $0$-chain and the partial boundaries of a $(0,0)$-chain vanish, moreover $D_0=\{(0,0)\}$, so we have the identities
\be\label{000}
\NN^G_{0,(0,0)}(\R^n)=\NN^G_0(\R^n)=\MM^G_0(\R^n)\qquad\text{ and }\qquad\NN^G_{0,0}(\R^n)=\MM^G_{0,0}(\R^n).
\ee
Hence, for $k=0$,~\eqref{j2}\&\eqref{j3} are the same and Theorem~\ref{thm_Ups_2} states that $\j_{0,0}$ defines an isomorphism from $\MM^G_0(\R^n)$ onto $\MM^G_{0,0}(\R^n)$. In fact we already know by~\cite[Theorem~5.6]{GM_tfc} that this isomorphism is an isometry from $(\MM^G_0(\R^n),\M)$ onto $(\MM^G_{0,0}(\R^n),\Mw)$.\medskip

\noindent
($*$) For the other extreme case $k=n$ ($k_1=n_1$ and $k_2=n_2$),  we have
\[
\FF^G_{n,(n_1,n_2)}(\R^n) =\FF^G_n(\R^n)=\MM^G_n(\R^n)\qquad\text{ and }\qquad\FF^G_{n_1,n_2}(\R^n)=\MM^G_{n_1,n_2}(\R^n),
\]
and we see that~\eqref{j1}\&\eqref{j2} are the same. Besides, these groups both identify isometrically with $L^1(\R^n,G)$ and performing these identifications, $\j_{n_1,n_2}$ is just the identity. We have the trivial fact that $\j_{n_1,n_2}\simeq \Id$ is an isomorphism from $\FF^G_n(\R^n)$ onto $\FF^G_{n_1,n_2}(\R^n)$ with moreover, for $A\in\FF^G_n(\R^n)$,
\be\label{M=M_n1n2}
\Fw(\j_{n_1,n_2}A)=\Mw(\j_{n_1,n_2}A)=\M(A)=\F(A).\medskip
\ee

\noindent
($*$) In the degenerate case $n_2=0$ (and symmetrically $n_1=0$), the definitions of the groups of flat chains and of tensor flat chains are the same. We have,
\[
\FF^G_{k,(k,0)}(\R^n)=\FF^G_k(\R^n)\st{\j_{k,0}=\Id}=\FF^G_{k,0}(\R^n)\qquad\text{when }n_2=0.
\] 

On the contrary, in all the other cases, the mappings~\eqref{j1}\&\eqref{j2} are not onto. Moreover, if $0<k<n$, the image of $\FF^G_{k,(k_1,k_2)}(\R^n)$ by $\j_{k_1,k_2}$ does not even cover the subgroup of rectifiable $(k_1,k_2)$-chains.
\begin{proposition}\label{prop_j_not_onto}
Assume that $G$ is not the trivial group and $n_1,n_2\ge1$.
\begin{enumerate}[(i)]
\item For $0<k<n$, there exist rectifiable tensor chains $A'\in\MM^G_{k_1,k_2}(\R^n)$ such that $A'\ne\j_{k_1,k_2}A$ for every chain $A\in\FF^G_{k,(k_1,k_2)}(\R^n)$.
\item There exist tensor chains $B'\in\FF^G_{0,0}(\R^n)$ such that $B'\ne\j_{0,0}B$ for every $B\in\FF^G_0(\R^n)$.
\end{enumerate}
\end{proposition}

\subsubsection*{Isometries}
Theorem~\ref{thm_Ups_2} does not say anything about the continuity of the inverse of~\eqref{j3}. 
On this matter, we expect that the identity 
\be\label{MtjM}
\Mw(\j_{k_1,k_2}A)=\Mt(A)\qquad\text{for every }A\in\FF^G_{k,(k_1,k_2)}(\R^n)
\ee
holds true, where $\Mt$ is the variant of $\M$ introduced in~\cite[Section~4.1]{GM_tfc} and obtained by substituting $\PP^G_{k_1,k_2}(\R^n)$ for $\PP^G_k(\R^n)$ in the definition of the mass. Denoting $\Nt(A):=\Mt(A)+\Mt(\pt A)$, we also expect that the identity
\be\label{NtjN}
\Nw(\j_{k_1,k_2}A)=\Nt(A)\qquad\text{for every }A\in\FF^G_{k,(k_1,k_2)}(\R^n)
\ee
holds, so that~\eqref{j3} would be an isometry from $(\NN^G_{k,(k_1,k_2)},\Nt)$ onto $(\NN^G_{k_1,k_2},\N)$.\\
Let us again consider the extreme values of $k$. When $k=n$, we have $\Mt=\M$ and we see from~\eqref{M=M_n1n2} that~\eqref{MtjM} holds true. For $k=0$ we also have $\Mt=\M$ and by~\cite[Theorem~5.6]{GM_tfc},~\eqref{MtjM} (and thus~\eqref{NtjN}) hold true for $k=k_1=k_2=0$. In the remaining cases, $0<k<n$, we are only able to establish the versions of~\eqref{MtjM}\&\eqref{NtjN} for the (coordinate) slicing masses $\M_{\Sl}$, $\Mw_{\Sl}$(Theorem~\ref{thm_MSlj=MSl} below). The definition of the slicing mass is postponed to Section~\ref{SD}. Let us say however that $\M_{\Sl}$ is equivalent to the usual mass. The precise result, established in the same section, is the following.

\begin{theorem}\label{thm_M_et_SlM} For $A\in\FF^G_k(\R^n)$ and $A'\in\FF^G_{k_1,k_2}(\R^n)$, there hold,
\[
\M(A)\le\M_{\Sl}(A)\le\sqrt{\binom nk}\,\M(A),\qquad\qquad\Mw(A')\le\Mw_{\Sl}(A')\le\sqrt{\binom{n_1}{k_1}\binom{n_2}{k_2}}\,\Mw(A').
\]
\end{theorem}
The proof of the inequalities $\M(A)\le\M_{\Sl}(A)$, $\Mw(A)\le\Mw_{\Sl}(A)$ is not immediate. For this we use the deformation theorems and Theorem~\ref{thm_restr} which extends to any chain or tensor chain the notion of restriction to any finite union of intervals (in fact the restrictions to almost all the translates of such set).\medskip

We can now state the advertised $\M_{\Sl}$-based isometry result.

\begin{theorem}\label{thm_MSlj=MSl}  For every $A\in\FF^G_k(\R^n)$ and every $(k_1,k_2)\in D_k$,
\begin{enumerate}[(i)]
\item~\qquad\qquad $\ds\sum_{(k'_1,k'_2)\in D_k} \Mw_{\Sl}(\j_{k'_1,k'_2}A)=\M_{\Sl}(A)$.
\item If $A\in\FF^G_{k,(k_1,k_2)}(\R^n)$, there hold
\[
\Mw_{\Sl}(\j_{k_1,k_2}A)=\M_{\Sl}(A)\qquad\text{ and }\qquad\Mw_{\Sl}(\pt_1\j_{k_1,k_2}A)+\Mw_{\Sl}(\pt_2\j_{k_1,k_2}A)=\M_{\Sl}(\pt A).
\] 
\end{enumerate}
\end{theorem}
Notice that the first identity in $(ii)$ is a direct consequence of $(i)$. Let us define the $\N$-like norms built on the slicing masses.  For chains $A$ and tensor chains $A'$, we set:
\[
\N_{\Sl}(A):=\M_{\Sl}(A)+\M_{\Sl}(\pt A)\qquad\text{ and }\qquad\Nw_{\Sl}(A'):=\Mw_{\Sl}(A')+\Mw_{\Sl}(\pt_1 A')+\Mw_{\Sl}(\pt_2 A').
\] 
By Theorem~\ref{thm_M_et_SlM}, these functions are norms, equivalent to the $\N$-norm in $\NN^G_k(\R^n)$ and to the $\Nw$-norm in $\NN^G_{k_1,k_2}(\R^n)$ respectively. As the second point of Theorem~\ref{thm_MSlj=MSl} implies that $\Nw_{\Sl}(\j_{k_1,k_2}A)=\N_{\Sl}(A)$ for $A\in\FF^G_{k,(k_1,k_2)}(\R^n)$, we obtain the following improvement of Theorem~\ref{thm_Ups_2}.
\begin{corollary}\label{coro_MSlj=MSl}~\\
The isomorphism~\eqref{j3} is an isometry from $(\NN^G_{k,(k_1,k_2)}(\R^n),\N_{\Sl})$ onto $(\NN^G_{k_1,k_2}(\R^n),\Nw_{\Sl})$.
\end{corollary}

\subsubsection*{Organization of the article}
\begin{enumerate}[($*$)]
\item  In Section~\ref{SB}, we start by recalling White's deformation theorem and then state the deformation theorem for tensor chains, Theorem~\ref{thm_defthm}. This result leads to a proof of Theorem~\ref{thm_Ups_2} (which states that~\eqref{j3} is onto).\smallskip

\noindent
The explicit constructions leading to the negative results of Proposition~\ref{prop_j_not_onto} are postponed to Subsection~\ref{Ss_ctrex} at the end of the article (the slicing mass is used in the proof).  
\item Section~\ref{SC} is dedicated to the proof of Theorem~\ref{thm_defthm}.
\item In Section~\ref{SD} we state and prove Theorem~\ref{thm_restr}. We also deduce Theorem~\ref{thm_supports} which states that for a chain $A$, if all the components of $\j A$ are supported in some closed set, then $A$ is also supported in this set.
\item In Section~\ref{SE}, we define the coordinate slicing mass and then prove Theorems~\ref{thm_M_et_SlM}\&\ref{thm_MSlj=MSl}. As already mentioned, the proof of Proposition~\ref{prop_j_not_onto} ends the article.
\end{enumerate}
Since the deformation theorem for tensor chains is not the main point of this note, its proof in Section~\ref{SC} may be skipped at first.

\section{The deformation theorem for tensor chains}
\label{SB}

\subsection{The deformation theorem for $k$-chains}
\label{SB1}
Before stating its generalization, let us recall the deformation theorem of~\cite{White1999-1} in the classical setting. First we introduce some notation about grids and provide some reminders about the operator $\chi$ and the notion of support of a chain.\medskip

\subsubsection*{Grids}
Let $\eps>0$.
\begin{enumerate}[(1)] 
\item We denote by $\GG^\eps$ the regular grid associated to the lattice $\eps\Z^n$. 
\item  For $y\in\R^n$,  $\tau_y$ is the translation defined by $\tau_y x=y+x$. The push-forward by $\tau_y$ of a flat chain or of a tensor flat chain $A$ in $\R^n$ is denoted $\tau_y A$.
\item The grid obtained by translating $\GG^\eps$ by the vector $(\eps/2) (1,1,\dots,1)$ (that is a half step in all directions) is called the dual grid of $\GG^\eps$ and is denoted ${\wh\GG}^\eps$. 
\item The $k$-faces of $\GG^\eps$ are oriented arbitrarily. Given a $k$-face $F$ with orientation $\x$, its dual face ${\wh F}$ is the $(n-k)$-face of ${\wh\GG}^\eps$ with same center and orientation $\wh \x$ such that $\x\we\wh\x=e_1\we e_2\we\dots\we e_n$. We denote $\la
\wh F\ra$ the affine $(n-k)$-plane spanned by $\wh F$.
 \end{enumerate}
 
\subsubsection*{The operator $\chi$}
Let us recall~\cite[Theorem~2.1]{White1999-2} that was restated  as~\cite[Theorem~5.1]{GM_tfc}. There exists a group morphism $\chi:\FF^G_0(\R^n)\to G$ such that,
\begin{enumerate}[(i)]
\item $\chi(\sum g_i\lb x_i\rb)=\sum g_i$, for any finite sequences $g_i\in G$ and $x_i\in\R^n$.
\item $|\chi (A)|_G\le\F(A)$ for every $A\in\FF^G_0(\R^n)$.
\end{enumerate}
 
\subsubsection*{Supports}
Eventually, we also recall that a chain $A$ (resp. a tensor chain $A'$) is supported in a closed set $S\sub\R^n$ if for any open neighborhood $U$ of $S$, there exists a sequence of polyhedral chains $Q_j$ (resp. of polyhedral tensor chains $Q'_j$) all supported in $U$ such that $Q_j\to A$ (resp. $Q'_j\to A'$).\medskip

We can now state the theorem (in fact only the parts that are relevant for our purpose).

\begin{theorem}[{\cite[Theorem~1.1]{White1999-1}}]\label{thm_defWhite} Let $\eps>0$, there exist two families of operators 
\[
P=P_\eps:\FF^G_k(\R^n)\longto \FF^G_k(\R^n),\qquad\qquad H=H_\eps:\FF^G_k(\R^n)\longto \FF^G_{k+1}(\R^n),
\] 
such that the following properties hold for $A,B\in \FF^G_k(\R^n)$ and almost every $y\in\R^n$.
\begin{enumerate}[(i)]
\item[(0)] The two following mappings are Lipschitz continuous group morphisms.
\[
\begin{array}{rcl}
\FF^G_k(\R^n)&\longto&L^1([0,\eps)^n,\FF^G_k(\R^n)),\\
C&\longmapsto&\qquad z\mapsto P(\tau_z C),
\end{array}
\qquad\quad
\begin{array}{rcl}
\FF^G_k(\R^n)&\longto&L^1([0,\eps)^n,\FF^G_{k+1}(\R^n)),\\
C&\longmapsto&\qquad z\mapsto H(\tau_z C).
\end{array}
\]
Moreover $z\mapsto \tau_{-z}P(\tau_z A)$ and $z\mapsto \tau_{-z}H(\tau_z A)$ are $\eps$-periodic along all the coordinate axes.
\item We have the identity $\pt P(\tau_y A)=P(\pt \tau_yA)$.
\item $\tau_yA=P(\tau_y A)+\pt H(\tau_y A)+H(\pt \tau_y A)$.
\item $P(\tau_y A)$ writes as $\sum g_F(y) F$ where the sum runs over the $k$-faces of $\GG^\eps$ and $g_F(y)\in G$ with $\sum|g_F(y)|_G<\oo$. Moreover, when $A$ is a polyhedral chain, 
\be\label{formula_gF}
g_F(y)=\chi\lt([(\tau_y A)\cap \la\wh F\ra]\restr \wh F\rt).
\ee
\item There exist a constant $c(n)\ge1$ only depending on $n$ such that, 
\begin{equation}\label{estimmassdefo}
\int_{(0,1)^n}\M\lt(P (\tau_{\eps z} A)\rt)\,dz\le c(n) \M(A),\qquad\qquad\int_{(0,1)^n}\M\lt(H (\tau_{\eps z} A)\rt)\,dz\le\eps c(n)\M(A).
\end{equation}
\item If $A$ is a polyhedral chain then $H(\tau_yA)$ is a polyhedral chain.
\item If $A$ is supported in $S$ then $P(\tau_yA)$ and $H(\tau_yA)$ are supported in $(S+y)+[-\eps/2,\eps/2]^n$.
\end{enumerate}
\end{theorem}

\begin{remark}\label{rem_point(0)}~
The point (0) is not explicit in~\cite{White1999-1} but follows from the construction. The fact that $y\mapsto P(\tau_y A)$ and $y\mapsto H(\tau_y A)$ belong to $L^1\loc$ allows us to give a meaning to the integrals of point (iv).
\end{remark}

\begin{remark}\label{rmk_morphismPH}
Let $(G^a,+,|\cd|_{G^a})$ and $(G^b,+,|\cd|_{G^b})$ be two complete Abelian normed groups and let $\phi:G^a\to G^b$ be a Lipschitz continuous group morphism. As described before Proposition~3.9 in~\cite{GM_tfc}, for $\b\in I^n$ such morphism induces a continuous morphism of chain complexes, 
\[
\phi_* :\FF_*(\X_\b,G^a)\longto \FF_*(\X_\b,G^b).
\]
We have the following identities, for every $\eps>0$ and every $A\in\FF_k(\R^n,G^a)$,
\begin{align*}
P_\eps(\phi_* \tau_y A)&= \phi_* [P_\eps(\tau_y A)]\qquad\text{in }L\loc^1\lt(\R^n,\FF_k(\R^n,G^b)\rt),\\
 H_\eps(\phi_* \tau_y A)&= \phi_* [H_\eps(\tau_y A)]\qquad\text{in }L\loc^1\lt(\R^n,\FF_k(\R^n,G^b)\rt).
\end{align*}
Indeed, the constructions of~\cite{White1999-1} do not depend on the particular group $G$. 
\end{remark}

In the proof of the deformation theorem, the following estimates are used, first with polyhedral chains. There hold for every $A\in \FF^G_k(\R^n)$ and $\eps>0$,
\begin{align}
\label{coro_def_White_estim_Av_PA} 
\int_{(0,1)^n}\F(P_\eps\tau_{\eps z}A)\,dz &\le c(n)\F(A),\\ 
\label{coro_def_White_estim_Av_HA} 
\int_{(0,1)^n}\F(H_\eps\tau_{\eps z}A)\,dz&\le c'(n,\eps)  \F(A).
\end{align}
These estimates yield the extension by continuity of the operators $P_\eps$ and $H_\eps$ (first defined on polyhedral chains). Since $P_\eps$ commutes with the boundary operator,~\eqref{coro_def_White_estim_Av_PA} is a direct consequence of the first estimate of Theorem~\ref{thm_defWhite}(iv) with the same constant.  On the contrary, $H_\eps$ does not commute with $\pt$ and the proof of~\eqref{coro_def_White_estim_Av_HA} involves a small trick which is not given in~\cite{White1999-1}.  Since we use~\eqref{coro_def_White_estim_Av_HA} in the proof of the ``tensor chain'' version of the deformation theorem, it is appropriate to present it here.

\begin{proof}[Proof of~\eqref{coro_def_White_estim_Av_HA}]
Let $\eps>0$ and let $Q\in \PP^G_k(\R^n)$ with a decomposition $Q=R+\pt S$,  $R\in\PP^G_k(\R^n)$, $S\in\PP^G_{k+1}(\R^n)$. We have for almost $y\in\R^n$,
\[
H_\eps(\tau_y Q)=H_\eps(\tau_y R)+H_\eps(\pt \tau_y S).
\] 
We bound the first term directly in average by using the second estimate of Theorem~\ref{thm_defWhite}(iv):
\be\label{coro_def_White_estim_Av_HA_1}
\int_{(0,1)^n}\F(H_\eps (\tau_{\eps z} R))\,dz\le\int_{(0,1)^n}\M(H_\eps (\tau_{\eps z} R))\,dz\le\eps c(n)\M(R).
\ee
For the second term, we face the issue  that in general $H_\eps\pt\ne\pt H_\eps$. The trick consists in ``integrating by parts''. Applying the identity  of Theorem~\ref{thm_defWhite}(ii) to $S$, we write
\[
H_\eps(\pt \tau_y S)=\tau_y S-P_\eps(\tau_y S)-\pt H_\eps(\tau_y S),
\]
which leads to
\[
\F(H_\eps(\pt \tau_y S))\le\M(S)+\M(P_\eps(\tau_y S))+\M (H_\eps(\tau_y S)).
\] 
Using the two estimates of Theorem~\ref{thm_defWhite}(iv), we obtain,
\be\label{coro_def_White_estim_Av_HA_2}
\int_{(0,1)^n}\F(H_\eps(\pt \tau_{\eps z} S))\,dz\le   (1+c(n)+\eps c(n))\M(S).
\ee
Summing~\eqref{coro_def_White_estim_Av_HA_1} and~\eqref{coro_def_White_estim_Av_HA_2}  and optimizing over the decompositions of $Q$, we get~\eqref{coro_def_White_estim_Av_HA} for polyhedral chains and with $c'(n,\eps)=1+c(n)+\eps c(n)$. We conclude with a density argument.
\end{proof}
Remark that for $\eps\le1$ we can take substitute for $c'(n,\eps)$ in~\eqref{coro_def_White_estim_Av_HA} the constant $c'(n):=2c(n)+1$ which does not depend on $\eps$.\medskip

Let us  recall the following essential result.
\begin{corollary}[{\cite[Corollary~1.2]{White1999-1}}]
\label{coro_defWhite}
There holds, for every $A\in\FF^G_k(\R^n)$,
\[
\int_{(0,1)^n}\F(P_\eps\tau_{\eps z} A-A)\,dz\ \st{\eps\dw0}\longto\ 0.
\]
\end{corollary}

\subsection{The deformation theorem for $(k_1,k_2)$-chains}
\label{SB2}
We propose a generalization of the deformation theorem to tensor chains.
\begin{theorem}\label{thm_defthm}
Let $\eps>0$, there exist three families of operators: 
\begin{align*}
\varPi^i=\varPi^i_\eps&: \FF^G_{k_1,k_2}(\R^n)\longto \FF^G_{k_1+i_1(i),k_2+i_2(i)}(\R^n),
\end{align*}
for $i=0,1,2$ and with
\[
i_1(0)=i_2(0)=0\qquad\quad\text{and}\qquad\quad i_l(i)=\d_{l,i}\quad\text{for }i,l=1,2,
\]
such that  for $A',B'\in \FF^G_{k_1,k_2}(\R^n)$ and almost every $y\in\R^n $ there hold:
\begin{enumerate}[(i)]
\item[(0)] For $i=0,1,2$ and $i_1(i),i_2(i)$ defined as above, we have the Lipschitz continuous group morphisms,
\[
\begin{array}{rcl}
\FF^G_{k_1,k_2}(\R^n)&\longto &L^1\lt([0,\eps)^n,\FF^G_{k_1+i_1(i),k_2+i_2(i)}(\R^n)\rt),\\
C'&\longmapsto&\qquad  z\mapsto \varPi^i(\tau_z C').
\end{array}
\]
Moreover the mappings $z\mapsto\tau_{-z}\varPi^i(\tau_z A')$ are $\eps$-periodic along all the coordinates axes. 
\item  We have 
\begin{align*}
 \pt_l \varPi^0(\tau_y A')&=\varPi^0\pt_l(\tau_yA')\quad\text{for }l=1,2,\\
 \pt_1 \varPi^2(\tau_y A')&=-\varPi^2\pt_1(\tau_yA'),\\
  \pt_2 \varPi^1(\tau_y A')&=-\varPi^1\pt_2(\tau_yA').
 \end{align*}
\item  We have the identity,
\begin{align*}
 \tau_yA'=\varPi^0(\tau_yA')+\pt_1\varPi^1(\tau_yA')+\pt_2\varPi^2(\tau_yA')+\varPi^1\pt_1(\tau_yA')+\varPi^2\pt_2(\tau_yA').
 \end{align*}
\item As in Theorem~\ref{thm_defWhite}, $\varPi^0(\tau_y A')$ writes as $\sum_Fg_F(y)F$ where $F$ runs over the set of $(k_1,k_2)$-faces of $\GG^\eps$ and $\sum_F|g_F(y)|<\oo$.  Moreover, when $A'\in\PP^G_{k_1,k_2}(\R^n)$ we have the same formula for the coefficients:
\be\label{formula_gprimeF}
g_F(y)=\chi\lt(\lt[(\tau_y A')\cap \la\wh F\ra\rt]\restr \wh F\rt).
\ee
\item There exist a constant $\wh c(n)\ge1$ only depending on $n$ such that
\[
 \int_{(0,1)^n}\M\lt(\varPi^0 (\tau_{\eps z}A')\rt)\,dz\le\wh c(n)\Mw(A'),
\]
and for $i=1,2$, 
\[
\int_{(0,1)^n}\Mw\lt(\varPi^i (\tau_{\eps z} A')\rt)\,dz\le\eps\wh c(n)\Mw(A').
\]
\item If $A'$ is a polyhedral tensor chain then for $i=1,2$, $\varPi^i(\tau_yA')$ is a polyhedral tensor chain.
\item If $A'$ is supported in $S$ then for $i=0,1,2$, $\Pi^i(\tau_yA')$ is supported in $(S+y)+[-\eps/2,\eps/2]^n$.
\end{enumerate}
\end{theorem}
\begin{remark}\label{rem_anticom}
The identities of point~(i) are crucial in the proof of Theorem~\ref{thm_Ups_2} below. Indeed, we can rewrite the identity of (ii) as 
\begin{multline*}
\tau_yA'=\varPi^0(\tau_yA')+\pt\varPi^1(\tau_yA')+ \pt\varPi^2(\tau_yA')+ \varPi^1\pt(\tau_yA')+ \varPi^2\pt(\tau_yA')\\
-\lt[\pt_2\varPi^1(\tau_yA')+\varPi^1\pt_2(\tau_yA')\rt]-\lt[\pt_1\varPi^2(\tau_yA')+\varPi^2\pt_1(\tau_yA')\rt].
\end{multline*}
Using the (anti-)commutation properties of~(i), we see that the two last terms in square brackets vanish and the identity simplifies to 
\begin{align}
\label{Theta_1}
\tau_yA'
&=\varPi^0(\tau_yA')+\pt\varPi^1(\tau_yA')+ \pt\varPi^2(\tau_yA')+ \varPi^1\pt(\tau_yA')+ \varPi^2\pt(\tau_yA')\\
\label{Theta_2}
&=\varPi^0(\tau_yA')+\pt \wh H(\tau_yA')+ \wh H \pt(\tau_yA'),
\end{align}
where we have set $\wh H:=\varPi^1+\varPi^2$.\\
In view of~\eqref{Theta_2} and the estimates and properties satisfied by the $\varPi^i$'s, we see that the pair $(\varPi^0,\wh H)$ behaves exactly as the pair $(P,H)$ in Theorem~\ref{thm_defWhite}, at least on polyhedral tensor chains. Since the latter are dense in the groups of flat chains, the pair $(\varPi^0,\wh H)$ (restricted to polyhedral tensor chains) extends by continuity on the groups $\FF^G_k(\R^n)$ as a pair $(\wh P,\wh H)$ which satisfies the same properties as $(P,H)$ in Theorem~\ref{thm_defWhite}. Moreover, as the formulas~\eqref{formula_gF}\&\eqref{formula_gprimeF} are the same on polyhedral tensor chains, we have in fact $\wh P=P$.
On the contrary, except in the degenerate cases $\{n_1,n_2\}=\{0,n\}$, we have $\wh H\ne H$. More importantly, the construction of~\cite{White1999-1} based on successive central projections\footnote{The method based on restrictions on $n$-cubes followed by successive central projections on $j$-faces for $j$ decreasing from $(n-1)$ to $k$ (and even $k-1$) dates back to the deformation theorem of Federer and Fleming~\cite{FedFlem60}.} mixes the coordinates $x^1,x^2$, so that in general  $H$ does not rewrite as $H^1+H^2$ with $H^1,H^2$ enjoying the same anti-commutation relations with respect to $\pt_2$, $\pt_1$ as $\varPi^1,\varPi^2$ in point~(i).
\end{remark}

Let us state the counterpart of Corollary~\ref{coro_defWhite}. As the proof is the same up to the obvious changes, we skip it.
\begin{corollary}
\label{coro_defthm}
There holds, for every $A'\in\FF^G_{k_1,k_2}(\R^n)$,
\[
\int_{(0,1)^n}\Fw(\varPi^0_\eps\tau_{\eps z} A'-A')\,dz\ \st{\eps\dw0}\longto\ 0.
\]
\end{corollary}

\subsection{Proof of Theorem~\ref{thm_Ups_2}}
\label{SB3}
Assuming that Theorem~\ref{thm_defthm} holds true, we are in position to prove Theorem~\ref{thm_Ups_2}.
Since $\j$ is one-to-one, we only need to establish that the group morphism~\eqref{j3} is onto.\medskip

Let $A'\in\NN^G_{k_1,k_2}(\R^n)$. From the identity~\eqref{Theta_1} we have for almost every $y\in\R^n$,
\[
\varPi_\eps^0\tau_yA'=\tau_yA'-\pt\lt(\lt[\varPi_\eps^1+\varPi_\eps^2\rt](\tau_yA')\rt)- \lt[\varPi_\eps^1\tau_y+\varPi_\eps^2\rt](\tau_y\pt A').
\]
Setting $\eps_i:=2^{-i}$ for $i\ge0$, we deduce for $j\ge i\ge0$ and almost every $y\in\R^n$,
\begin{multline}
\label{proof_thm_Ups_2_1}
\varPi_{\eps_j}^0\tau_yA'-\varPi_{\eps_i}^0\tau_yA'
=\pt\lt(\lt[\varPi_{\eps_j}^1+\varPi_{\eps_j}^2-\varPi_{\eps_i}^1-\varPi_{\eps_i}^2\rt](\tau_yA')\rt)\\
+\lt[\varPi_{\eps_j}^1+\varPi_{\eps_j}^2-\varPi_{\eps_i}^1-\varPi_{\eps_i}^2\rt](\tau_y\pt A').
\end{multline}
Now we notice that by construction and also from the formula~\eqref{formula_gprimeF} we have $\varPi_{\eps_j}^0\tau_{y'}Q=Q$ for almost every $y'\in(-\eps_j/2,\eps_j/2)^n$ and for any tensor chain of the form $Q=\sum_Fg_FF$ where the sum runs over the oriented $(k_1,k_2)$-faces $F$ of $\GG^{\eps_j}$ and where $\sum_F|g_F|<\oo$. Since $\eps_i/\eps_j$ is an integer, both $\varPi_{\eps_j}^0(\tau_yA')$ and $\varPi_{\eps_i}^0(\tau_yA')$ are of this form, applying $\varPi_{\eps_i}^0\tau_{y'}$ to~\eqref{proof_thm_Ups_2_1} we get 
\begin{multline*}
\varPi_{\eps_j}^0\tau_yA'-\varPi_{\eps_i}^0\tau_yA'
=\pt\varPi_{\eps_j}^0\tau_{y'}\lt(\lt[\varPi_{\eps_j}^1+\varPi_{\eps_j}^2-\varPi_{\eps_i}^1-\varPi_{\eps_i}^2\rt](\tau_yA')\rt)\\
+\varPi_{\eps_j}^0\tau_{y'}\lt(\lt[\varPi_{\eps_j}^1+\varPi_{\eps_j}^2-\varPi_{\eps_i}^1-\varPi_{\eps_i}^2\rt](\tau_y\pt A')\rt),
\end{multline*}
where we used the relation $\varPi_{\eps_j}^0\pt=\pt\varPi_{\eps_j}^0$. We deduce
\begin{multline*}
\F\lt(\varPi_{\eps_j}^0\tau_yA'-\varPi_{\eps_i}^0\tau_yA'\rt)\le
\M\lt(\varPi_{\eps_j}^0\tau_{y'}\lt(\lt[\varPi_{\eps_j}^1+\varPi_{\eps_j}^2-\varPi_{\eps_i}^1-\varPi_{\eps_i}^2\rt](\tau_yA')\rt)\rt)\\
+\M\lt(\varPi_{\eps_j}^0\tau_{y'}\lt(\lt[\varPi_{\eps_j}^1+\varPi_{\eps_j}^2-\varPi_{\eps_i}^1-\varPi_{\eps_i}^2\rt](\tau_y\pt A')\rt)\rt).
\end{multline*}
Writting $y'=\eps_j z'$ and integrating over $z'\in(-1/2,1/2)^n$, we get using the first estimate of Theorem~\ref{thm_defthm}(iv) and the $\eps$-periodicity of $y'\mapsto \tau_{-y'}\varPi^0_\eps\tau_{y'}$.
\begin{multline*}
\F\lt(\varPi_{\eps_j}^0\tau_yA'-\varPi_{\eps_i}^0\tau_yA'\rt)\\
\le\wh c(n)\sum_{\eps\in\{\eps_i,\eps_j\}} \sum_{l=1}^2
\lt[\Mw\lt(\varPi_\eps^l(\tau_yA')\rt)+\Mw\lt(\varPi_\eps^l(\tau_y\pt_1 A')\rt)+\Mw\lt(\varPi_\eps^l(\tau_y\pt_2 A')\rt)\rt].
\end{multline*}
Integrating with respect to $y\in(0,1)^n$ and using the second estimate of Theorem~\ref{thm_defthm}(iv) together with the $\eps_i$-periodicity of the operators $y\mapsto \tau_{-y}\varPi^l_{\eps_i}\tau_y$ for $l=1,2$, we get 
\[
\int_{(0,1)^n}\F\lt(\varPi_{\eps_j}^0\tau_yA'-\varPi_{\eps_i}^0\tau_yA'\rt)\,dy\le2[\wh c(n)]^2(\eps_i+\eps_j)\Nw(A')\le4[\wh c(n)]^2\eps_i\Nw(A').
\]
Let us now set for $y\in(0,1)^n$ and $i\ge0$, $Q_i(y):=\tau_{-y}\varPi_{\eps_j}^0\tau_{y}A'$, the preceding estimate rewrites as
\[
\int_{(0,1)^n}\F\lt(Q_j(y)-Q_i(y)\rt)\,dy\le4[\wh c(n)]^2\eps_i\Nw(A')=2^{2-i}[\wh c(n)]^2\Nw(A').
\]
We deduce that $Q_j$ is a Cauchy sequence in $L^1((0,1)^n,\FF^G_k(\R^n))$. Let us denote $y\mapsto A(y)$ its limit. Taking into account the $\eps_i$-periodicity of $y\mapsto Q_i(y)$, the estimates of Theorem~\ref{thm_defthm}(iv) and Fatou's lemma lead to
\begin{align*}
\int_{(0,1)^n}\M(A(y))\,dy&\le\liminf\int_{(0,1)^n}\M(Q_j(y))\,dy\le\wh c(n)\Mw(A'),\\
\int_{(0,1)^n}\M(\pt A(y))\,dy&\le\liminf\int_{(0,1)^n}\M(\pt Q_j(y))\,dy\le\wh c(n)\Mw(\pt A'),
\end{align*}
where for the second estimate, we have used $\pt\varPi^0_\eps=\varPi^0_\eps\pt$. In particular, $A(y)\in\NN^G_k(\R^n)$ for almost every $y\in(0,1)^n$.

Next, by Theorem~\ref{thm_defthm}(iii), each $Q_j(y)$ lies in $\PP^G_{k_1,k_2}(\R^n)$. Consequently $\j_{k'_1,k'_2}Q_j(y)=0$ for $(k'_1,k'_2)\ne(k_1,k_2)$. By continuity of $\j$ we get that $\j_{k'_1,k'_2}A(y)=0$ for $(k'_1,k'_2)\ne(k_1,k_2)$ so that $A(y)\in\FF^G_{k,(k_1,k_2)}(\R^n)$ and thus $A(y)\in\NN^G_{k,(k_1,k_2)}(\R^n)$ for almost every $y\in\R^n$. 

Now we consider the tensor chain $Q_j'(y):=Q_j(y)=\tau_{-y}\varPi^0_{\eps_i}\tau_yA'$ as an element of $\FF^G_{k_1,k_2}(\R^n)$. Using again the $\eps_j$-periodicity of $y\mapsto Q'_j(y)$ we write for $j\ge1$,
\begin{align}
\nonumber
\int_{(0,1)^n}\Fw\lt(Q'_j(y)-A'\rt)\, dy
&=\int_{(0,1)^n}\Fw\lt(Q'_j(\eps_jz)-A'\rt)\, dz\\
\nonumber
&=\int_{(0,1)^n}\Fw\lt(\varPi^0_{\eps_j}\tau_{\eps_jz}A'-\tau_{\eps_jz}A'\rt)\, dz\\
\label{proof_thm_Ups_2_2}
&\le\int_{(0,1)^n}\Fw\lt(\varPi^0_{\eps_j}\tau_{\eps_jz}A'-A'\rt)\, dz+ \int_{(0,1)^n}\Fw\lt(A'-\tau_{\eps_jz}A'\rt)\, dz.
\end{align}
By Corollary~\ref{coro_defthm}, the first term of~\eqref{proof_thm_Ups_2_2} goes to 0 as $j\up\oo$. For the second term, we fix $z\in(0,1)^n$ and introduce the Lipschitz continuous mapping
\[
h:(t,x)\in[0,1]\times\R^n\longmapsto t \eps_jz+x\in\R^n.
\]
Denoting $(e_0,e_1,\dots,e_n)$ the standard basis of $\R\times\R^n$ we write the homotopy formula
\[
\tau_{\eps_jz}A'-A'=\pt h\pf\lt(\lb(0,e_0)\rb\we A'\rt)+h\pf\lt(\lb(0,e_0)\rb\we\pt A'\rt).
\]
This yields 
\[
\Fw(A'-\tau_{\eps_jz}A)\le \sqrt{n}|z|\eps_j\Nw(A').
\]
We deduce that the second term of~\eqref{proof_thm_Ups_2_2} also goes to 0 as $j\up\oo$ so that the sequence $y\mapsto Q'_j(y)$ converges towards $y\mapsto A'$ in $L^1((0,1)^n,\FF^G_{k_1,k_2}(\R^n))$.\\
Passing to the limit in the identity $\j_{k_1,k_2}Q_j(y)=Q'_j(y)$ we get by continuity of $\j_{k_1,k_2}$  that $\j_{k_1,k_2}A(y)=A'$ for almost every $y\in(0,1)^n$. Since $A(y)\in\NN^G_{k,(k_1,k_2)}(\R^n)$ this proves that $\j_{k_1,k_2}:\NN^G_{k,(k_1,k_2)}(\R^n)\to\NN^G_{k_1,k_2}(\R^n)$ is onto, hence  the result. We may notice that as $\j$ is one-to-one, $y\mapsto A(y)$ is constant.\hfill$\square$

\section{Proof of Theorem~\ref{thm_defthm}}
\label{SC}
The construction leading to the deformation of tensor chains is based on the deformation theorem for ``classical'' chains, Theorem~\ref{thm_defWhite} (denoted DT in the proof below). Considering a tensor chain $A'\in\FF^G_{k_1,k_2}(\R^n)$, we apply the DT to the $k_1$-chain  $\i A'\in\FF_{k_1}(\X_\a,\FF^G_{k_2}(\X_\ova))$. We obtain a polyhedral  $k_1$-chain with coefficients in $\FF^G_{k_2}(\X_\ova)$. In a second step, we apply the DT on these coefficients. 

By scaling, we can assume $\eps=1$. \medskip

\noindent
\textit{Step 1 (Definition of the operators on $\PP^G_{k_1,k_2}(\R^n)$).}\\  
Let $Q'\in\PP^G_{k_1,k_2}(\R^n)$, we apply the DT to $\i Q'$. Notice that by Remark~\ref{rmk_morphismPH} applied to the $1$-Lipschitz inclusion, \[\phi:(\MM^G_{k_2}(\R^n),+,\M)\hookrightarrow(\FF^G_{k_2}(\R^n),+,\F),\]
we see that whether we consider $\i Q'$ as a chain with coefficients in $\MM^G_{k_2}(\R^n)$ or in $\FF^G_{k_2}(\R^n)$ does not matter for the definition of $P(\tau_{y^1}\i Q')$ or $H(\tau_{y^1}\i Q')$.\\
We have for almost every $y^1\in\X_\a$,
\[
\tau_{y^1}\i Q'=P(\tau_{y^1}\i Q')+\pt H(\tau_{y^1}\i Q')+H(\pt \tau_{y^1}\i Q').
\]
Obviously, $\tau_{y^1}\i Q'=\i\tau_{y^1} Q'$ and since $\i$ is a group isomorphism, we can write
\[
\tau_{y^1}Q'=\i^{-1}(P(\i \tau_{y^1} Q'))+\i^{-1}(\pt H(\i \tau_{y^1}Q'))+\i^{-1}(H(\pt \i \tau_{y^1} Q')).
\]
Defining 
\be\label{PiPi1}
\varPi := \i^{-1}P\i, \qquad\quad\varPi^1:=\i^{-1} H \i,
\ee
and using $\pt\i=\i\pt_1$, we get,
\be\label{A_defthm_1}
\tau_{y^1} Q'=\varPi(\tau_{y^1}Q')+\pt_1\varPi^1(\tau_{y^1} Q')+\varPi^1(\tau_{y^1}\pt_1 Q').
\ee
Notice that by~DT(iii)\&(v), 
\[
\varPi:\PP^G_{k_1,k_2}(\R^n)\longto\PP^G_{k_1,k_2}(\R^n)\quad\text{and}\quad\varPi^1:\PP^G_{k_1,k_2}(\R^n)\longto\PP^G_{k_1+1,k_2}(\R^n).
\]
Moreover, for almost every $y^1\in\X_\a$, 
\be\label{A_defthm_2}
\varPi (\tau_{y_1} Q')=\sum_{F^1}F^1\we g_{F^1}(y^1)\qquad\text{ where }\ g_{F^1}(y^1)\in \MM^G_{k_2}(\X_\ova), 
\ee
and the sum runs over the set of oriented $k_1$-faces of the grid with vertices $\Z^{n_1}$ in $\X_\a$. \\
Let now $y^2\in \X_\ova$ and  $y:=y^1+y^2$. Before taking the push-forward of~\eqref{A_defthm_1} by $\tau_{y^2}$, let us observe that we have the following relations. 
\be\label{A_defthm_3}
\tau_{y^2}\pt_1=\pt_1\tau_{y^2},\quad\tau_{y^2}\varPi=\varPi\tau_{y^2},\quad \tau_{y^2}\varPi^1=\varPi^1\tau_{y^2},\quad \pt_2\varPi=\varPi\pt_2,\quad \pt_2\varPi^1=-\varPi^1\pt_2.
\ee
The first identity is obvious and the others follow from Remark~\ref{rmk_morphismPH} applied to the Lipschitz continuous group morphisms, 
\[
\tau_{y_2}:\FF^G_{k_2}(\X_\ova)\to\FF^G_{k_2}(\X_\ova)\qquad\text{ and }\qquad\pt_2:\FF^G_{k_2}(\X_\ova)\to\FF^G_{k_2-1}(\X_\ova).
\]
We only detail the proof of the last identity (and explain the  presence of the minus sign). Recall that for a $(k'_1,k'_2)$-chain $R'=gp^1\we p^2$, we have by definition $\pt_2R'=(-1)^{k'_1}g p^1\we(\pt p^2)$. Arguing exactly as in \cite[ (5.9)]{GM_tfc} and using Remark~\ref{rmk_morphismPH}
we have for $Q'\in \PP^G_{k_1,k_2}(\R^n)$, 
\be\label{ipt2Q'}
\i(\pt_2 Q')= (-1)^{k_1}\pt_* \i Q'.
\ee
Taking into account that  the operator $\varPi^1$ raises the first index of tensor chains (if $Q'\in \FF^G_{k_1,k_2}(\R^n)$ then $\varPi^1 Q'\in \FF^G_{k_1+1,k_2}(\R^n)$), we compute, 
\[\begin{array}{rll}
\i(\pt_2 \varPi^1Q')&\st{\eqref{ipt2Q'}}  =(-1)^{k_1+1}\ov\pt\pf \i(\varPi^1 Q') 
    &\st{\eqref{PiPi1}}=(-1)^{k_1+1}\ov\pt\pf H\i Q'\\
&\st{\text{(Rem.~\ref{rmk_morphismPH})}}=(-1)^{k_1+1}H\ov\pt\pf\i Q'
     &\st{\eqref{ipt2Q'}}=(-1)^{k_1+1}(-1)^{k_1} H\i(\pt_2 Q')
          \st{\eqref{PiPi1}}=-\i \varPi^1\pt_2Q'.
\end{array}
\]
The last identity of~\eqref{A_defthm_3} then follows by applying $\i^{-1}$. \medskip

Let us resume the construction. Composing~\eqref{A_defthm_1} by $\tau_{y^2}$ and using the identities~\eqref{A_defthm_3}, we get for almost every $y\in\R^n $,
\be\label{A_defthm_4}
\tau_y Q'=\varPi(\tau_yQ')+\pt_1\varPi^1(\tau_yQ')+\varPi^1(\pt_1\tau_yQ').
\ee
Moreover, we deduce from $\pt P=P\pt$ that 
\be\label{A_defthm_4.5}
\pt_1\varPi(\tau_yQ')= \varPi\pt_1(\tau_yQ').
\ee

Composing by $\tau_{y^2}$ the expression~\eqref{A_defthm_2} of $\varPi(\tau_{y^1}Q')$, we get with obvious notation,
\[
\varPi (\tau_yQ')=\sum_{F^1} F^1\we g_{F^1}(y).
\]
Applying the DT to the chains $g_{F^1}(y)\in \PP^G_{k_2}(\X_\ova)$, we set
\be\label{Pi0Pi2}
\varPi^0(\tau_yQ'):=\sum_{F^1} F^1\we P(g_{F^1}(y)),\qquad\quad \varPi^2(\tau_yQ'):=(-1)^{k_1}\sum_{F^1} F^1\we H(g_{F^1}(y)),
\ee
By DT(ii), there holds
\[
\varPi (\tau_yQ')=\varPi^0(\tau_yQ')+\pt_2\varPi^2(\tau_yQ')+\varPi^2(\pt_2\tau_yQ').
\]
Substituting this identity in~\eqref{A_defthm_4} we obtain the desired relation for every $Q'\in\PP^G_{k_1,k_2}(\R^n)$ and almost every $y\in\R^n $, namely: 
\be\label{A_defthm_5}
\tau_yQ'=\varPi^0(\tau_yQ')+\pt_1\varPi^1(\tau_yQ')+\varPi^1(\pt_1\tau_yQ')+\pt_2\varPi^2(\tau_yQ')+\varPi^2(\pt_2\tau_yQ').
\ee
 Using again $\pt P=P\pt$ and~\eqref{A_defthm_4.5},  we also have the identities
\[
\pt_l \varPi^0(\tau_yQ')=\varPi^0 \pt_l(\tau_yQ')\quad\text{for }l=1,2\qquad\text{and}\qquad\pt_1\varPi^2(\tau_yQ')=-\varPi^2 \pt_1(\tau_yQ').
\]
The minus sign in the last identity comes from the factor $(-1)^{k_1}$ in the definition of $\varPi^2$. At this point, provided that $A'$ and $B'$ are polyhedral tensor chains, the operators $\varPi^i$, $i=0,1,2$ comply to points (i),(ii),(iii) of the theorem except formula~\eqref{formula_gprimeF}. We also observe that the points~(v)\&(vi) also hold by construction from the corresponding points of the DT.\medskip

Let us establish~\eqref{formula_gprimeF}. We decompose a $(k_1,k_2)$-face $F$ of $\GG^1\cap \X_\a$ as $F=F^1\we F^2$ where $F^1$ is $k_1$-face of the grid $\GG^1\cap \X_\a$ and $F^2$ a $k_2$-faces of $\GG^1\cap \X_\ova$. \\
By construction and point~(iii) of the DT, the coefficient $g_{F^1}(y^1)$ in~\eqref{Pi0Pi2} is given by 
\[
g_{F^1}(y^1)=\chi\lt([(\tau_{y^1} \i Q')\cap \la\wh F^1\ra]\restr \wh F^1\rt),
\]
where $\wh F^1$ is the face dual to $F^1$ in $\X_\a$ and $\la\wh F^1\ra$ is the affine space spanned by $\wh F^1$. Then, again by construction and point~(iii) of  DT, $P(g_{F^1}(y))$ writes $\sum_{F^2} g'_{F^1\we F^2}(y)F^2$ where, with similar notation, 
\[
g'_F(y)=\chi\lt(\lt[(\tau_{y^2} [g_{F^1}(y^1)])\cap \la\wh F^2\ra\rt]\restr \wh F^2\rt).
\]
Using the commutation properties of restrictions and slicing and the obvious fact that $\chi(\chi(\i R'))=\chi(R')$ for a $0$-polyhedral chain, this formula simplifies to
\[
g'_F(y)=\chi\lt(\lt[(\tau_yQ')\cap\la\wh F\ra\rt]\restr \wh F\rt)\qquad\text{for almost every }y\in\R^n.
\]
This is formula~\eqref{formula_gprimeF}. 

In summary, the properties~(i)(ii)(iii)(v)\&(vi) of the theorem hold for polyhedral tensor chains. 
Moreover the operators $\varPi^0,\varPi^1,\varPi^2$ are group morphisms.\medskip

\noindent
\textit{Step 2 (Estimates and extension of the operators).}
Let $Q'\in \PP^G_{k_1,k_2}(\R^n)$. We first estimate the average masses of $\varPi(\tau_y Q')$ and $\varPi^1(\tau_y Q')$. From the estimates~DT(iv) and the fact that $\i$ is an isometry from $(\PP^G_{k_1,k_2}(\R^n),\M)$ onto $\ds\lt(\PP_{k_1}\lt(\X_\a,(\PP_{k_2}^G(\X_\ova),\M\rt),\M\rt)$ (see \cite[Section 4.6]{GM_tfc}), we deduce, 
\begin{align}
\label{A_defthm_8}
\int_{\X_\a\cap (0,1)^n} \M(\varPi \tau_{y^1}Q')\,dy^1&\le c(n_1)\M(Q'),\\
\label{A_defthm_9}
\int_{\X_\a\cap (0,1)^n}\M(\varPi^1 \tau_{y^1} Q')\,dy^1&\le c(n_1)\M(Q').
\end{align}
Since $\varPi$ and $\varPi^1$ commute with $\tau_{y^2}$ for $y^2\in\X_\ova$, we can take integrate over $\X_\ova\cap(0,1)^n$ and get the same estimates. 

Next, using the left formula of~\eqref{Pi0Pi2} and the estimates~of DT(iv) in $\X_\ova$, we obtain for almost every $y^1\in \X_\a$,
\[
\int_{\X_\a\cap (0,1)^n}\M(\varPi^0\tau_{y^1+y^2} Q')\,dy^2\le c(n_2) \sum \M(F) \M(g_F(y^1))=c(n_2) \M(\varPi\tau_{y^1} Q').
\]
Integrating over $y^1\in \X_\a\cap(0,1)^n$ and using~\eqref{A_defthm_8} yield, 
\be\label{A_defthm_10}
\int_{(0,1)^n}\M(\varPi^0\tau_yQ')\,dy\le c(n_1)c(n_2) \M(Q').
\ee
Similarly, using the right formula of~\eqref{Pi0Pi2} and~\eqref{A_defthm_8} we obtain, 
\be\label{A_defthm_11}
\int_{(0,1)^n}\M(\varPi^2\tau_yQ')\,dy\le c(n_1)c(n_2)\M(Q').\medskip
\ee

Let us now establish that the operators $\varPi^i$ extend by continuity to all tensor chains. The proof is similar to the one in~\cite{White1999-1} for the operators $P$,$H$.\\
The first step consists in showing the continuity (in average) of these operators in $\Fw$-norm. Let us decompose $Q'\in \PP^G_{k_1,k_2}(\R^n)$ as $Q'=R^{0,0}+\pt_1R^{1,0}+\pt_2R^{0,1}+\pt_1\pt_2 R^{1,1}$. We have by~\eqref{A_defthm_5}, 
\[
\varPi^0 \tau_yQ'= \tau_y \varPi^0R^{0,0}+\pt_1\varPi^0 \tau_y R^{1,0}+\pt_2\varPi^0 \tau_yR^{0,1}+\pt_1\pt_2 \varPi^0 \tau_yR^{1,1},
\]
hence $\Fw(\varPi^0\tau_y Q')\le\sum\M(\varPi^0 \tau_y R^{i_1,i_2})$. Averaging over $y\in(0,1)^n$, using~\eqref{A_defthm_10} and optimizing over the decompositions, we obtain the desired estimate:
\be\label{A_defthm_12}
\int_{(0,1)^n}\Fw(\varPi^0\tau_yQ')\,dy\le c(n_1)c(n_2)\Fw(Q').
\ee
The case of the operators $\varPi^1$, $\varPi^2$ is slightly more difficult because $\varPi^1$ does not commute with $\pt_1$ and $\varPi^2$ does not commute with $\pt_2$. 
We first treat the former.
Applying~\eqref{coro_def_White_estim_Av_HA} to $\i Q'$, we get
\[
\int_{\X_\a\cap(0,1)^n}\Fw(\varPi^1\tau_{y^1}Q')\,dy^1=\int_{\X_\a\cap(0,1)^n}\Fw(H\tau_{y^1}\i Q')\,dy^1\le    c'(n_1,1)\F(\i Q')= c'(n_1,1)\Fw(Q').
\]
Averaging over $y^2\in\X_{\ov a}\cap(0,1)^n$ yields
\be\label{A_defthm_13}
\int_{(0,1)^n}\Fw(\varPi^1\tau_yQ')\,dy\le c'(n_1,1)\Fw(Q').
\ee
 For the operator $\varPi^2$, we come back to the  definition~\eqref{Pi0Pi2} of $\varPi^2\tau_yQ'$. We compute, using the DT in $\FF^\Z_{k_1}\lt(\X_\a,\FF^G_{k_2}(\X_\ova)\rt)$ and the fact that $\h^{k_1}(F^1)=1$ for every $k_1$-face of the grid,
\begin{align}
\nonumber 
\int_{(0,1)^n}\M\lt(\i \varPi^2 (\tau_y Q')\rt)\,dy&\le\sum_{F^1}\int_{(0,1)^n}  \F\lt(H\lt(\tau_{y^2}g_{F^1}(y^1)\rt)\rt)\,dy \\
\nonumber 
&\st{\eqref{coro_def_White_estim_Av_HA}}\le  c'(n_2,1) \sum_{F^1}\int_{(0,1)^n\cap\X_a}\F\lt(g_{F^1}(y^1)\rt)\,dy^1 \\
\nonumber 
&=c'(n_2,1) \int_{(0,1)^n\cap\X_a}\sum_{F^1}\h^{k_1}(F^1)\F\lt(g_{F^1}(y^1)\rt)\,dy^1.
\end{align}
Now by the expression \eqref{A_defthm_2} of $\varPi \tau_{y^1} Q'$ we have 
\begin{align}
\nonumber
\int_{(0,1)^n\cap\X_a}\sum_{F^1}\h^{k_1}(F^1)\F\lt(g_{F^1}(y^1)\rt)\,dy^1&= \int_{(0,1)^n\cap\X_a}\M(\i \varPi  \tau_{y^1} Q')\,dy^1\\
\nonumber 
&= \ \int_{(0,1)^n\cap \X_\a} \M( P\i \tau_{y^1}Q')\,dy^1\\
\nonumber
&\st{\eqref{estimmassdefo}}\le c(n_1) \M(\i Q').
\end{align}
Therefore,
\begin{equation}\label{A_defthm_135}
 \int_{(0,1)^n}\M\lt(\i \varPi^2 (\tau_y Q')\rt)\,dy\le c'(n_2,1)c(n_1)\M(\i Q').
\end{equation}

Next, let us decompose $\i Q'$ as $\i Q'=R+\pt S$. Denoting $R':=\i^{-1}R\in\PP^G_{k_1,k_2}(\R^n)$, $S':=\i^{-1}S\in\PP^G_{k_1+1,k_2}(\R^n)$, we write  $Q'=R'+\pt_1 S'$. Using $\pt_1\varPi^2\tau_y=\varPi^2\pt_1\tau_y$, we  have for almost every $y$,
\[
\varPi^2 (\tau_yQ')=\varPi^2 (\tau_yR')+\pt_1\varPi^2 (\tau_yS').
\] 
Hence, for almost every $y\in \R^n$,
\[
\Fw\lt(\varPi^2 (\tau_yQ')\rt)=\F\lt(\i\varPi^2 (\tau_yQ')\rt)\le\M\lt(\i \varPi^2 (\tau_yR')\rt)+\M\lt(\i \varPi^2 (\tau_yS')\rt).
\]
Applying~\eqref{A_defthm_135} to $R'$ and $S'$ yields
\[
\int_{(0,1)^n}\Fw\lt(\varPi^2 (\tau_yQ')\rt)\,dy\le c'(n_2,1)c(n_1)\lt(\M(R)+\M(S)\rt).
\]
Eventually, optimizing over the decompositions $\i Q'=R+\pt S$, we get
\be\label{A_defthm_14}
\int_{(0,1)^n}\Fw\lt(\varPi^2 (\tau_yQ')\rt)\, dy\le c'(n_2,1)c(n_1)\F(\i Q')= c'(n_2,1)c(n_1)Fw(Q').
\ee

The estimates~\eqref{A_defthm_12}\eqref{A_defthm_13}\&\eqref{A_defthm_14} then allow us to extend by continuity the operators $Q'\mapsto (y\mapsto \varPi^i\tau_y Q')$ as mappings,
\[
\begin{array}{rcl}
\varPi^i:\FF^G_{k_1,k_2}(\R^n)&\longto& L^1\loc\lt(\R^n,\FF^G_{k_1+i_1(i),k_2+i_2(i)}(\R^n)\rt),\\
A&\longmapsto&\qquad y\mapsto \varPi^i(\tau_yA).
\end{array}
\]
Moreover, by construction, for $i=0,1,2$, $ y\mapsto \varPi^i(\tau_yA)$ is $1$-periodic in all the coordinate directions. This establishes the point~(0). With the estimates~\eqref{A_defthm_12}\eqref{A_defthm_13}\&\eqref{A_defthm_14} we can pass to the limit in the identities of points~(i) and (ii). The first part of the property~(iii) holds true by continuity. The formula~\eqref{formula_gprimeF} in~(iii) and point~(v) only concern polyhedral tensor chains and have been already established. The property~(vi) on the supports of the tensor chains $\varPi^i(\tau_yA')$ also extends by a diagonal argument. Eventually passing to the limit in~\eqref{A_defthm_9}\eqref{A_defthm_10}\&\eqref{A_defthm_11} and using the lower semicontinuity of the mass, we obtain the estimates of point~(iv). This achieves the proof of the theorem.\hfill$\square$

\section{Restriction of tensor chains to translates of figures}
\label{SD}
In the proof of Proposition~\ref{prop_SlM=MC} in the next section, we consider the restriction on union of intervals of chains and tensor chains which do not have necessarily finite masses. The validity and meaning of this operation are provided by the theorem below. In its statement and it its proof the generalized restriction of $A'$ on the set $x+J$ is denoted $[A']_J(x)$. In the rest of the article, we will use the notation $A'\srestr (x+J)$, which is closer to the notation used for standard restrictions.

We say that $I\sub\R^n$ is an interval of $\R^n$ if $I=I_1\t I_1\t\cdots\t I_n$ where the $I_j$'s are intervals of $\R$ (bounded or unbounded, open, semi-open or closed). A figure is a finite union of intervals. Observe that the set of figures is an algebra (but not a $\sigma$-algebra), it is stable by any intersection, finite union, and by taking complements.

\begin{theorem}\label{thm_restr}
Let $J\sub\R^n$ be a figure. The mapping,
\begin{align*}
\MM^G_{k_1,k_2}(\R^n)&\longto L^1\loc\big(\R^n,\FF^G_{k_1,k_2}(\R^n)\big),\\
A'&\longmapsto\quad x\mapsto A'\restr(x+J),
\end{align*}
extends as a group morphism,
\begin{align*}
\FF^G_{k_1,k_2}(\R^n)&\longto L^1\loc(\R^n,\FF^G_{k_1,k_2}),\\
A'&\longmapsto\quad x\mapsto [A']_J(s).
\end{align*}
This operator enjoys the following properties for any $A',B'\in\FF^G_{k_1,k_2}(\R^n)$, any figures $J$, $\wt J$ and almost every $x\in\R^n$.
\begin{enumerate}[(i)]
\item $[A']_{\R^n}(x)=A'$.
\item $[A']_{J\cap \wt J}(x)=\ds\lt[[A']_J(x)\rt]_{\wt J}(x)$.
\item If $J\cap \wt J=\void$, there holds  $[A']_{J\cup \wt J}=[A']_J(x)+[A']_{\wt J}(x)$.
\item for almost every $y\in\R^n$, $[A']_{y+J}(x)=[A']_J(x+y)$.
\item If $A'$ is supported in $S$ then $[A']_J(x)$ is supported in $S\cap(x+\ov J)$.
\item If $A'_j\in\FF^G_{k_1,k_2}(\R^n)$ converges to $A'$, then $[A'_j]_{J}\to [A']_J\ \text{ in }L^1\loc\big(\R^n,\FF^G_{k_1,k_2}(\R^n)\big)$.
\item As $\ell\up\oo$, we have $\ds\lt(x\mapsto [A']_{\R^n\sm(-\ell,\ell)^n}(x)\rt)\longto0$ in $L^1\loc\big(\R^n,\FF^G_{k_1,k_2}(\R^n)\big)$.
\end{enumerate}
\end{theorem}

The construction is based on~\cite[Lemma~4.5(i), estimate~(4.9)]{GM_tfc}. This estimate is rephrased as follows.
\begin{lemma}\label{lem_restr}
Let  $Q'_j$ be a sequence of elements of $\PP^G_{k_1,k_2}(\R^n)$, for any $I\sub\R^n$ interval and for any Cartesian product $\om=\om_1\t \om_2\t\cdots\t\om_n\sub\R^n$ such that  for $1\le i\le n$, $\om_i\sub\R$ is measurable with finite length, there holds
\[
\int_\om\sum\Fw\lt(Q'_j\restr(x+I)\rt)\,dx\le c_\om\sum\Fw(Q'_j),
\]
where $c_\om\ge 0$ depends on $\om$.
\end{lemma}
In the proof below a set $\om$ of the above form is called a \emph{fvcp} (for ``finite volume cartesian product'').

\begin{proof}[Proof of Theorem~\ref{thm_restr}]~

\noindent
\textit{Step 1. Definition of the restriction operator on intervals.}

Let $A'\in\FF^G_{k_1,k_2}(\R^n)$ and let $Q'_j\in\PP^G_{k_1,k_2}(\R^n)$ such that $Q'_j\to A'$ rapidly (that is $\sum \Fw(Q'_j-A)<\oo$). Denoting $R'_j:=Q'_{j+1}-Q'_j$, we have $\sum \Fw(R'_j)<\oo$. We even assume, 
\[
\Fw(Q'_1)+\sum \Fw(R'_j)\le2\Fw(A').
\]
Applying Lemma~\ref{lem_restr} to $(Q'_1,R'_1,R'_2,\dots)$, we get for any interval $I$ and any \emph{fvcp} $\om$,
\[
\int_\om\lt[\Fw(Q'_1\restr (x+I))+\sum \Fw(R'_j\restr (x+I))\rt]\,dx\le 2c_\om\Fw(A')<\oo.
\]
Setting, 
\[
[A']_I(x):= Q'_1\restr (x+I)+\sum R'_j\restr (x+I),
\]
the series converges normally in $L^1\loc(\R^n,\FF^G_{k_1,k_2}(\R^n))$. Moreover, the limit does not depend on $Q'_j$ and we have the estimate,
\be\label{thm_restr_1}
\int_\om\sum\Fw\lt([A']_I(x)\rt)\,dx\le 2c_\om\sum\Fw(A').
\ee 
Remark that by definition, for finite mass tensor chains, $[A']_I(x)=A\restr(x+I)$. \medskip

\noindent
\textit{Step 2. Properties of the restriction operator on intervals.}

Let $I$ be an interval and let $I_r$ be a finite partition of $I$ in intervals. Since we have $Q'\restr I=\sum Q'\restr I_r$ for every polyhedral chain $Q'$, it follows from the construction that, 
\be\label{thm_restr_2}
[A']_I(x)=\sum [A']_{I_r}(x)\qquad\text{for almost every }x\in\R^n.
\ee
Similarly, if $I$, $I'$ are two intervals, there holds
\be\label{thm_restr_3}
[A']_{I\cap I'}(x)=\lt[[A']_I(x)\rt]_{I'}(x).
\ee
\medskip

\noindent\textit{Step 3. Definition and properties of the restriction operator.}

Let $J$ be a figure and let $I_r$ be a finite partition of $J$ in intervals. We define:
\[
[A']_J(x):=\sum [A']_{I_r}(x).
\] 
Using subpartitions and~\eqref{thm_restr_2}, we see that this definition does not depend on the partition $I_r$.\medskip
 
Let us check the properties stated in the theorem. First, points~(i)\&(iv) are obvious by construction. Point~(ii) follows from~\eqref{thm_restr_3} and~(iii) from~\eqref{thm_restr_2}.\medskip

For point~(v) we assume that $A'$ is supported in some closed set $S$. For any open set $U\supset S$ we can choose the $Q'_j$ all supported in $U$ so that $Q'_j\restr(x+J)$ is supported in $U\cap (x+\ov J)$. Passing to the limit we get that $[A']_J(x)$ is supported in $S\cap (x+\ov J)$. This proves~(v).\medskip

To establish~(vi), we may assume that $J=I$ is an interval. Let $A'_j\to A'$ and let $\om\sub\R^n$ be a \emph{fvcp}. By~\eqref{thm_restr_1}, we have, 
\be\label{thm_restr_4}
\int_\om\Fw\lt([A'_j-A']_I(x)\rt)\,dx\le 2c_\om\Fw(A'_j-A')\ \to\ 0.
\ee
Moreover, $[A'_j]_I(x)-[A']_I(x)=[A'_j-A']_I(x)$ for almost every $x$. Hence,~\eqref{thm_restr_4} implies that $[A'_j]_I(x)\to[A']_I(x)$ in $L^1\loc\big(\R^n,\FF^G_{k_1,k_2}(\R^n)\big)$. This proves~(vi).\medskip

Eventually, let again $\om\sub\R^n$ be a~\emph{fvcp} and assume moreover that $\om$ is bounded. For $\ell>0$ we set $D_\ell:=\R^n\sm(-\ell,\ell)^n$. Let $Q'_j\in\PP^G_{k_1,k_2}(\R^n)$ such that $Q'_j\to A'$.  On the one hand, writing $\R^n$ as the disjoint union of $(-\ell,\ell)^n$ and $D_\ell$ we have by~(i)\&(iii):
\[
[Q'_j-A']_{D_\ell}(x)=(Q'_j-A')-[Q'_j-A']_{(-\ell,\ell)^n}(x)\qquad\text{ for almost every }x\in\R^n.
\]
We deduce from~\eqref{thm_restr_1},
\[
\int_\om\Fw\lt([Q'_j-A']_{D_\ell}(x)\rt)\,dx\le\lt(\h^k(\om)+2c_\om\rt)\Fw(Q'_j-A').
\]
As a consequence,
\be\label{thm_restr_5}
\int_\om\Fw\lt([Q'_j-A']_{D_\ell}(x)\rt)\,dx\to0\qquad\text{ as }j\up\oo\text{, uniformly with respect to }\ell.
\ee
On the other hand for every $j$, $Q'_j$ is compactly supported, hence (recall that $\om$ is bounded) $Q'_j\restr (x+D_\ell)$ vanishes for every $x\in \om$ and for $\ell$ large enough. We write 
\[
[A']_{D_\ell}(x)
=[Q'_j]_{D_\ell}(x)+[Q'_j-A']_{D_\ell}(x)=[Q'_j-A']_{D_\ell}(x)
\]
for $x\in\om$ and $\ell$ large enough, depending on $j$. 
With~\eqref{thm_restr_5} and a diagonal argument, we get
\[
\int_\om\Fw\lt([A']_{D_\ell}(x)\rt)\,dx\to 0\qquad\text{as }\ell\up\oo.
\]
This proves~(vii) and the theorem.
\end{proof}
\begin{remark}~\label{rem_restr}
\begin{enumerate}[(1)]
\item When $A$ has finite mass, we have $A'_J(x)=A'\restr(x+J)$. As announced earlier, in the sequel, we use the notation 
\[
A'\srestr(x+J):=[A']_J(x),
\]
 which, despite property~(iv), is slightly ambiguous.
\item Notice that by making $(n_1,n_2)=(n,0)$, the lemma holds true for flat chains.
\item The mapping $J\longmapsto\lt(x\mapsto A\srestr(x+J)\rt)$ defines a \emph{finitely additive} measure on the algebra of figures with values in $L^1\loc(\R^n,\FF^G_{k_1,k_2}(\R^n))$. 
\item We can complete the list of properties of the theorem by:
\begin{enumerate}
\item[\textit{(viii)}] If $\h^k(J)=0$ then $A'\srestr(x+J)=0$ (for almost every $x\in\R^n$).
\end{enumerate}
\end{enumerate}
\end{remark}
As a corollary, we obtain the following result about supports of chains.
\begin{theorem}\label{thm_supports}
Let $A\in\FF^G_k(\R^n)$ and $S\sub\R^n$ be a closed set. If for every $(k'_1,k'_2)\in D_k$,  $\j_{k'_1,k'_2}A$ is supported in $S$  then $A$ is supported in $S$. 
\end{theorem}

\begin{proof}
Let $A\in\FF^G_k(\R^n)$. We claim that for every figure $J$, every $(k'_1,k'_2)\in D_k$ and almost every $x\in\R^n$,
\be\label{thm_supports_1}
(\j_{k'_1,k'_2}A)\srestr(x+J)= \j_{k'_1,k'_2}[A\srestr(x+J)].
\ee
This relation is true when $A$ is a finite mass chain. Applying this to a sequence of finite mass chains converging to $A$ and passing to the limit, we obtain by Theorem~\ref{thm_restr}(vi) and continuity of $\j$ that~\eqref{thm_supports_1} holds in the he general case.  \medskip

Now, let $S\sub\R^n$ be a closed set an assume that the tensor chains $\j_{k'_1,k'_2}A$ are all supported in $S$.  Let $U$ be an open neighborhood of $S$. We have to establish that $A$ is the limit of a sequence of polyhedral chains all supported in $U$.\medskip

\noindent
\textit{Step 1.} We first assume that $S$ is compact. There exists a relatively compact open set $V$ such that $S\sub V\sub \ov V \sub U$. Since the collection of open intervals whose closure is contained in $U$ covers $\ov V$, there exists an open figure $J$ such that $\ov V\sub J\sub \ov J\sub U$. Therefore, there exists $\eps>0$ such that $V\sub x+J\sub U$ for every $x\in B_\eps$.\\
Let us fix $(k'_1,k'_2)\in D_k$. By assumption, there exists a sequence $Q'_j\in\PP^G_{k'_1,k'_2}(\R^n)$ all supported in $V$ such that $Q'_j\to \j_{k'_1,k'_2}A$. Denoting $J^c$ the complement of $J$ in $\R^n$, for every $x\in B_\eps$, we have $x+J^c\sub \R^n\sm V$, hence  $Q'_j\restr(x+J^c)=0$. Sending $j$ to $\oo$, we get from Theorem~\ref{thm_restr}(vi), that for almost every $x\in B_\eps$,
\[
0=[\j_{k'_1,k'_2}A]\srestr(x+J^c)\st{\eqref{thm_supports_1}}=\j_{k'_1,k'_2}[A\srestr(x+J^c)].
\]
Since this holds true for every $(k'_1,k'_2)\in D_k$, we deduce from the injectivity of $\j$ that $A\srestr(x+J^c)$ vanishes for almost every $x\in B_\eps$. By Theorem~\ref{thm_restr}(i)\&(iii), we get for almost every $x\in B_\eps$,
\[
A=A\srestr\R^n=A\srestr(x+J)+A\srestr(x+J^c)=A\srestr(x+J),
\]
and by Theorem~\ref{thm_restr}(v), we see that $A$ is supported in $x+J\sub U$. The result is established if $S$ is compact.\medskip

\noindent
\textit{Step 2.} In the general case, we set for  $\ell>0$ and  $x\in\R^n$, 
\[
A_\ell(x):=A\srestr (x+(-\ell,\ell)^n),\qquad\qquad B_\ell(x):=A\srestr \lt(x+\R^n\sm(-\ell,\ell)^n\rt).
\]

\noindent
By Theorem~\ref{thm_restr}(iii) we have 
\be\label{thm_supports_2}
A=A_\ell(x)+B_\ell(x)\quad\text{for almost every }x.
\ee
We use Theorem~\ref{thm_restr}(vii) to treat the last term. We have
\be\label{thm_supports_3}
\int_{(0,1)^n}\F\lt(B_\ell(x)\rt)\,dx\to 0\qquad \text{as }\ell\up\oo.
\ee
Next, by~\eqref{thm_supports_1}, there holds for almost every $x\in\R^n$,
\[
j_{k'_1,k'_2}A_\ell(x)=(\j_{k'_1,k'_2}A)\srestr (x+(-\ell,\ell)^n),
\]
and by Theorem~\ref{thm_restr}(v), we get that for almost every $x\in\R^n$ and every $(k'_1,k'_2)\in D_k$ the tensor chain $j_{k'_1,k'_2}A_\ell(x)$ is supported in $S\cap (x+[-\ell,\ell]^n)$ which is a compact subset of $S$. By Step~1, we get that $A_\ell(x)$ is supported in $S$ for almost every $x\in\R^n$. That is:
\begin{multline}\label{thm_supports_4}
\text{There exists a sequence } Q_j^{(\ell)}\in\PP^G_{k_1,k_2}(\R^n)\text{ with }\supp Q_j^{(\ell)}\sub U\text{ such that}\\
Q_j^{(\ell)}\to A_\ell(x)\text{ for every }\ell>0\text{ and almost every }x\in\R^n.
\end{multline}

Using a diagonal argument, we deduce from~\eqref{thm_supports_2}\eqref{thm_supports_3}\&\eqref{thm_supports_4}  that there exist sequences $x_i\in\R^n$ and $\ell_i\up\oo$ such that for every $i$, $\wh Q_i:=Q^{\ell_i}_{j(\ell_i)}(x_i)\in\PP^G_k(\R^n)$ is supported in $U$ and $\wh Q_i\to A$. This proves the theorem.
\end{proof}

\section{Slicing mass}
\label{SE}
Let us define the (coordinate) slicing mass.
\begin{definition}\label{def_SlM}
The (coordinate) slicing mass of a $k$-chain $A$ or of a $(k_1,k_2)$-chain $A'$ in $\R^n$ is given by,
\[
\M_{\Sl}(A):=\sum_{\g\in I^n_k} \int_{\X_\g} \M(A\cap\X_{\ov\g}(x) )\,dx,\qquad\qquad
\Mw_{\Sl}(A'):=\!\!\!\!\!\sum_{\substack{\g\in I^n_k,\\(|\g^1|,|\g^2|)=(k_1,k_2)}}\!\!\!\!\! \int_{\X_\g} \Mw(A'\cap\X_{\ov\g}(x) )\,dx.
\]
\end{definition}
For flat chains and as a corollary of the deformation theorem of White, if $\M_{\Sl}(A)=0$ then $A=0$, see~\cite[Theorem~3.2]{White1999-2}. However, it was not clear whether
\[
\M_{\Sl}(A)<\oo\implies\M(A)<\oo.
\]
Theorem~\ref{thm_M_et_SlM} below settles this issue: the (coordinate) slicing mass is equivalent to the usual mass on the whole group of (tensor) chains. We obtain this result \textit{via} the identities of Proposition~\ref{prop_SlM=MC} which involve the coordinate masses $\M_{\CC}$. These latter are defined as the usual mass, except that we substitute for $\PP^G_k(\R^n)$ the subgroup $\CP^G_k(\R^n)$ and for $\PP^G_{k_1,k_2}(\R^n)$ the subgroup $\CP^G_{k_1,k_2}(\R^n)$).

\begin{definition} We set,
\begin{align*}
\CP^G_k(\R^n)
&:=\big\{Q\in\PP^G_k(\R^n):\text{for }0\le j\le n\text{ and for each }j\text{-face }F\text{ of }\supp Q\\
&\phantom{:=\big\{Q\in\PP^G_k} 
\text{the affine $j$-space spanned by }F\text{ is of the form } \X_\beta(x)  \text{ with } |\beta|=j \big\},\\
\\
\CP^G_{k_1,k_2}(\R^n)
&:=\CP^G_k(\R^n)\cap\TP^G_{k_1,k_2}(\R^n)\\
&=\big\{Q\in\CP^G_k(\R^n):\text{each }k\text{-face of }\supp Q\text{ is of the form }F^1\we F^2,\\
&\phantom{=\big\{Q\in\CP^G_k(\R^n):\text{each }k\text{-face}}
\text{with }F^1\text{ a }k_1\text{-face of }\X_\a\text{ and }F^2\text{ a }k_2\text{-face of }\X_\ova\big\}.
\end{align*}
\end{definition}
Notice that as a consequence of the deformation theorems, $\CP^G_k(\R^n)$ is dense in $\FF^G_k(\R^n)$ and $\CP^G_{k_1,k_2}(\R^n)$ is dense in $\FF^G_{k_1,k_2}(\R^n)$.

\begin{definition} 
The coordinate mass of $A\in\FF^G_k(\R^n)$ or $A'\in\FF^G_{k_1,k_2}(\R^n)$ is then given by:
\begin{align*}
\M_{\CC}(A)&:=\inf\lt\{\liminf \M(Q_j):Q_j\in\CP^G_k(\R^n),\ Q_j\st{\F}\to A\rt\},\\
\Mw_{\CC}(A')&:=\inf\lt\{\liminf \M(Q'_j):Q'_j\in\CP^G_{k_1,k_2}(\R^n),\ Q'_j\st{\Fw}\to A'\rt\}.
\end{align*}
\end{definition}
By definition we have $\M_{\CC}\ge \M$ and $\Mw_{\CC}\ge \Mw$.  It turns out that these coordinate masses coincide with the coordinate slicing masses.
\begin{proposition}\label{prop_SlM=MC}
For every $A\in\FF^G_k(\R^n)$  and every $A'\in\FF^G_{k_1,k_2}(\R^n)$,
\[
\M_{\Sl}(A)=\M_{\CC}(A),\qquad\qquad\Mw_{\Sl}(A')=\Mw_{\CC}(A').
\]
\end{proposition}

In the proof, we use the operatror $\chiw$ introduced in~\cite[Definition~5.3]{GM_tfc} which extends $\chi$ to $(0,0)$-chains. Let us recall its definition.
\[
\chiw(A'):=\chi(\chi(\i A')),\qquad\text{for }A'\in\FF^G_{0,0}(\R^n).
\] 
As stated in \cite[Proposition~5.4]{GM_tfc} we have:
\begin{align}
\label{form_chiw_1} 
|\chiw(A')|_G\le\Fw(A')\le\Mw(A')\qquad\text{for every }A'\in\FF^G_{0,0}(\X_\b),\\
\label{form_chiw_2}
\chi(A)=\chiw(\j_{0,0}A)\qquad\text{for every }A\in\FF^G_0(\X_\b).
\end{align}
In particular, $\chiw$ is 1-Lipschitz continuous. Also notice that for a polyhedral $0$-chain, we have $\j_{0,0}A=A$ hence $\chiw(A)=\chi(A)$.

\begin{proof}[Proof of Proposition~\ref{prop_SlM=MC}]
Since flat chains are a particular case of tensor chains with $\{n_1,n_2\}=\{0,n\}$, we only consider these latter. Let $A'\in\FF^G_{k_1,k_2}(\R^n)$.

For any coordinate polyhedral $(k_1,k_2)$-chain $Q$, there holds $\M_{\Sl}(Q')=\M(Q')$. Now for $A'\in\FF^G_{k_1,k_2}(\R^n)$, taking a sequence $Q'_j\in\PP^G_{k_1,k_2}(\R^n)$ such that $Q'_j\to A'$ and $\M(Q'_j)\to\M_{\CC}(A')$, we deduce by continuity of slicing and lower semi-continuity of the mass, that 
\be\label{prop_SlM=MC_1}
\Mw_{\Sl}(A')\le\Mw_{\CC}(A').
\ee
Let us establish the opposite inequality. Without loss of generality we assume that $\Mw_{\Sl}(A')<\oo$. By Theorem~\ref{thm_defthm}(iii), we have for $\eps>0$ and almost every $y\in\R^n$ that $\varPi^0_\eps\tau_yA'$ is a limit in mass\footnote{Actually $\varPi^0_\eps\tau_yA'\in \CP^G_{k_1,k_2}(\R^n)$ if $A'$ has compact support.} of elements of $\CP^G_{k_1,k_2}(\R^n)$. We deduce from Corollary~\ref{coro_defthm} that, 
\be\label{prop_SlM=MC_2}
\Mw_{\CC}(A')\le\liminf_{\eps\dw0 }\int_{(0,1)^n}\M(\varPi^0_\eps\tau_{\eps z}A')\,dz.
\ee
Now, for $\eps>0$ fixed and $Q'\in\PP^G_{k_1,k_2}(\R^n)$, we have by formula~\eqref{formula_gprimeF}, for almost every $y\in\R^n$,
\be\label{prop_SlM=MC_3}
\varPi^0_\eps\tau_yQ'=\sum_FF\we\chi\lt([(\tau_yQ')\cap \la\wh F\ra)]\restr \wh F\rt),
\ee
where the sum runs over the $k$-faces of $\GG^\eps$. Let $Q'_j\in\PP^G_{k_1,k_2}(\R^n)$ such that $Q'_j\to A'$. Applying~\eqref{prop_SlM=MC_3} with $Q'=Q'_j$ and sending $j\to\oo$, we get, by Theorem~\ref{thm_restr}(vii) and continuity of slicing, $\chiw$ and $\varPi^0_\eps$,
\begin{equation}\label{extension}
\varPi^0_\eps\tau_yA'=\sum_FF\we\chiw\lt([(\tau_yA')\cap \la\wh F\ra)]\srestr \wh F\rt).
\end{equation}
Using estimate~\eqref{form_chiw_1} in the form $|\chiw(B')|_G\le\Mw(B')$ for $B'\in\FF^G_{0,0}(\R^n)$, this leads to
\[
\int_{(0,1)^n}\M(\varPi^0_\eps\tau_{\eps z}A')\,dz\le\int_{(0,1)^n}\sum_F\h^k(F)\Mw\lt((\tau_{\eps z}A')\cap \la\wh F\ra)\srestr \wh F\rt)\,dz.
\]
Recalling that $\Mw_{\Sl}(A')<\oo$, the $0$-slices $(\tau_{\eps z}A')\cap \la\wh F\ra$ have finite $\Mw$-mass for almost every $z\in(0,1)^n$ and we can rewrite the previous inequality as, 
\be\label{prop_SlM=MC_4}
\int_{(0,1)^n}\M(\varPi^0_\eps\tau_{\eps z}A')\,dz\le\int_{(0,1)^n}\sum_F\h^k(F)\Mw\lt((\tau_{\eps z}A')\cap \la\wh F\ra)\restr \wh F\rt)\,dz.
\ee
Since the collection of dual faces $\wh F$ is a partition of the union of:
\begin{enumerate}[($*$)]
\item the affine $(n-k)$-plane $\{0\}^k\t\R^{n_k}$, 
\item its translates by $\eps m$ for $m\in\Z^k$,
\item all the affine planes obtained from these former by permuting the coordinates, 
\end{enumerate}
we deduce by Fubini that the right-hand side of~\eqref{prop_SlM=MC_4} is equal to $\Mw_{\Sl}(A')$. Hence,
\[
\int_{(0,1)^n}\M(\varPi^0_\eps\tau_{\eps z}A')\,dz\le\Mw_{\Sl}(A').
\]
With~\eqref{prop_SlM=MC_2} we get $\Mw_{\CC}(A')\le\Mw_{\Sl}(A')$ and recalling~\eqref{prop_SlM=MC_1}, we conclude that $\Mw_{\Sl}(A')=\Mw_{\CC}(A')$.
\end{proof}

\begin{remark}\label{rem_SlM=MC}~

\noindent
(a) In the above proof, the role of the operator $\srestr$ of Theorem~\ref{thm_restr} is essential. Indeed, taking the mass in~\eqref{prop_SlM=MC_3} before passing to the limit leads to,
\[
\int_{(0,1)^n}\M(\varPi^0_\eps\tau_{\eps z}Q'_j)\,dz\le\Mw_{\Sl}(Q'_j).
\]
Sending $j$ to $\oo$ and then $\eps$ to 0 yield $\Mw_{\CC}(A')\le\liminf\Mw_{\Sl}(Q'_j)$. To conclude, we need a sequence of polyhedral tensor chains $Q'_j$ such that $Q'_j\to A$ and $\Mw_{\Sl}(Q'_j)\to\Mw_{\Sl}(A')$. However at this point of the demonstration we do not know whether such sequence exists. Let us mention that if $2\le k\le n-2$, we do not have  
\[
Q_j\to A\text{ with }\M(Q_j)\to\M(A)\quad\ \implies\ \quad\M_{\Sl}(Q_j)\to\M_{\Sl}(A).
\]

\noindent 
(b) Notice that \eqref{extension} extends the formula~\eqref{formula_gprimeF} to $\FF_{k_1,k_2}^G(\R^n)$.\medskip

\noindent
(c) It follows from the proof that:
\[
\M_{\Sl}(A)=\lim_{\eps\dw0}\int_{(0,1)^n}\M(P_\eps\tau_{\eps z}A)\,dz,
\qquad\qquad
\Mw_{\Sl}(A')=\lim_{\eps\dw0}\int_{(0,1)^n}\Mw(\varPi^0_\eps\tau_{\eps z}A')\,dz.
\smallskip
\]

\noindent
(d) The identities of Proposition~\ref{prop_SlM=MC} concern the anisotropic ($\ell^1$-like) $\Mw_{\CC}$ mass. We conjecture that it has a counterpart with the usual mass. Namely, for $A\in \MM^G_k(\R^n)$, 
\[
\int_{\Gr(k,\R^n)} \int_\X \M(A\cap(x+\X^\perp))\,d\h^k(x)\,d\nu(\X)=c(n,k) \M(A),
\]
where $\Gr(k,\R^n)$  is the Grassmannian of $k$-planes in $\R^n$ and $\nu$ is the Borel probability measure on $\Gr(k,\R^n)$ invariant by the action of $\OO(\R^n)$. Such identity holds true for rectifiable chains as a consequence of \cite[3.2.26]{Federer} (see also \cite[(27)]{DePH} or \cite[Lemma 3.1]{CdRMS2017}) and the constant $c(n,k)$ can be determined by taking for $A$ an oriented $k$-disk in $\R^n$.\medskip

\noindent
(e) For $\g\in I^n_k$ we define the group of polyhedral $\g$-chain as,
\[
\CP^G_\g(\R^n):=\lt\{Q\in\CP^G_k(\R^n) : \text{each }j\text{-face of }\supp Q\text{ spans some }\X_\beta(x)\text{ with }\beta\sub\g\rt\}.
\]
Denoting $\Pi_\g$ the natural projection of $\CP^G_k(\R^n)$ on $\CP^G_\g(\R^n)$, we set, 
\begin{align*}
\M_{\CC,\g}(A)&:=\inf\lt\{\liminf \M(\Pi_\g Q_j):Q_j\in\CP^G_\g(\R^n),\ Q_j\st{\F}\to A\rt\},\\
\Mw_{\CC,\g}(A')&:=\inf\lt\{\liminf \M(\Pi_\g Q_j):Q_j\in\CP^G_\g(\R^n),\ Q_j\st{\Fw}\to A'\rt\}.
\end{align*}
We have, for the same reasons as~\eqref{prop_SlM=MC_1},
\begin{align*}
\M_{\Sl,\g}(A)&:=\int_{\X_\g} \M(A\cap\X_{\ov\g}(x) )\,dx\le\M_{\CC,\g}(A),\\
\Mw_{\Sl,\g}(A')&:=\int_{\X_\g} \Mw(A'\cap\X_{\ov\g}(x))\,dx\le\Mw_{\CC,\g}(A').
\end{align*}
By definition,
\[
\M_{\Sl}(A)=\sum_\g\M_{\Sl,\g}(A),\qquad\qquad\Mw_{\Sl}(A')=\sum_\g  \Mw_{\Sl,\g}(A'),
\]
and we obtain,
\[
\M_{\Sl}(A)\le\sum_\g\M_{\CC,\g}(A)\le\M_{\CC}(A),\qquad\qquad\Mw_{\Sl}(A')\le\sum_\g\Mw_{\CC,\g}(A')\le\Mw_{\CC}(A').
\]
The identities of the proposition and~(c) then lead to
\begin{align*}
\M_{\Sl,\g}(A)&=\M_{\CC,\g}(A)=\lim_{\eps\dw0}\int_{(0,1)^n}\M(\Pi_\g P_\eps\tau_{\eps z}A)\,dz,\\
\Mw_{\Sl,\g}(A')&=\Mw_{\CC,\g}(A')=\lim_{\eps\dw0}\int_{(0,1)^n}\M(\Pi_\g\varPi^0_\eps\tau_{\eps z}A')\,dz.
\end{align*}
\end{remark}

With the preceding result, we easily establish the equivalence of the mass and of the slicing mass as stated in Theorem~\ref{thm_M_et_SlM}.

\begin{proof}[Proof of Theorem~\ref{thm_M_et_SlM}]
Let $A\in\FF^G_k(\R^n)$ and $A'\in\FF^G_{k_1,k_2}(\R^n)$. The identities of Proposition~\ref{prop_SlM=MC} yield
\[
\M(A)\le\M_{\CC}(A)=\M_{\Sl}(A),\qquad\qquad\Mw(A')\le\Mw_{\CC}(A')=\Mw_{\Sl}(A'),
\]
which are half of the the inequalities of the theorem.
We are left to establish: 
\be\label{prf_thm_M_et_SlM_0}
\M_{\Sl}(A)\le\sqrt{\binom nk}\,\M(A),\qquad\qquad\Mw_{\Sl}(A')\le\sqrt{\binom{n_1}{k_1}\binom{n_2}{k_2}}\,\Mw(A').
\ee
By lower semicontinuity of the slicing mass with respect to $\F$- (respectively $\Fw$-) convergence, it is sufficient to establish these inequalities for polyhedral $k$-chains (respectively for polyhedral $(k_1,k_2)$-chains). Localizing by using restrictions, we may even assume that $A$ is of the form $gp$ for some polyhedral $k$-cell $p$  (respectively  that $A'$ is of the form $gp'$ for a polyhedral $(k_1,k_2)$-cell $p'$). In this case, denoting $\pi_\g$ the orthogonal projection on $\X_\g$, we have the obvious identities
\be\label{prf_thm_M_et_SlM_1}
\int_{\X_\g} \M\lt((gp)\cap\X_{\ov\g}(x)\rt)\,d\h^k(x)=|g|_G\int_{\X_\g} \M\lt(p\cap\X_{\ov\g}(x)\rt)\,d\h^k(x) =|g|_G\M(\pi_\g p),
\ee
and similarly
\be\label{prf_thm_M_et_SlM_2}
\int_{\X_\g} \M\lt((gp')\cap\X_{\ov\g}(x)\rt)\,d\h^k(x) =|g|_G\M(\pi_\g p').
\ee
Let us establish the first inequality of~\eqref{prf_thm_M_et_SlM_0}. Summing~\eqref{prf_thm_M_et_SlM_1} over $\g\in I^n_k$ we obtain
\[
\M_{\Sl}(gp)=|g|_G\sum_{\g\in I^n_k} \M(\pi_\g p).
\]
Then, using the Cauchy--Schwarz inequality together with the fact that $I^n_k$ has cardinal $\binom nk$ and then using the Cauchy--Binet formula, we compute:
\[
\M_{\Sl}(gp)
\le|g|_G\sqrt{\binom nk\sum_\g\M(\pi_\g p)^2\,}
=|g|_G\sqrt{\binom nk}\,\M(p)=\sqrt{\binom nk}\,\M(pg).
\]
This yields the first inequality of~\eqref{prf_thm_M_et_SlM_0} for $A=gp$ and thus in the general case.\\
For the second inequality, we remark that since $p'$ is a $(k_1,k_2)$-cell, there holds $\pi_\g p'=0$ for $\g\in I^n_k\sm J$ where $J:=\{\g\in I^n_k:(|\g^1|,|\g^2|)=(k_1,k_2)\}$. Hence, summing~\eqref{prf_thm_M_et_SlM_2} over $\g\in J$, we obtain  
\[
\M_{\Sl}(gp')=|g|_G\sum_{\g\in J}\M(\pi_\g p').
\]
Performing the same computations as above we obtain the second inequality of~\eqref{prf_thm_M_et_SlM_0} for $A'=gp'$, the only difference is that the cardinal of $J$ is $\binom{n_1}{k_1}\binom{n_2}{k_2}$.
\end{proof}

Let us now derive the identities of Theorem~\ref{thm_MSlj=MSl} (which involve $\j$ and the slicing masses). 
\begin{proof}[Proof of Theorem~\ref{thm_MSlj=MSl}]
Let $A\in\FF^G_k(\R^n)$,  let $(k'_1,k'_2)\in D_k$ and let $\g\in I^n_k$ such that $(|\g^1|,|\g^2|)=(k'_1,k'_2)$. Using the commutation rule for slicing and $\j$ (see \cite[Proposition 4.13]{GM_tfc}), we have for almost every $x\in\X_\g$,
\[
(\j_{k'_1,j'_2}A)\cap\X_{\ov\g}(x)=\j_{0,0}\lt(A\cap\X_{\ov\g}(x)\rt).
\]
This identity involves $(0,0)$-chains and by~\cite[Theorem~5.6]{GM_tfc}, there holds for almost every $x\in\X_\g$,
\[
\Mw\lt((\j_{k'_1,j'_2}A)\cap\X_{\ov\g}(x)\rt)=\Mw\lt(\j_{0,0}\lt[A\cap\X_{\ov\g}(x)\rt]\rt)=\M(A\cap\X_{\ov\g}(x)).
\]
Integrating over $\X_\g$, we get,
\[
\int_{\X_\g}\Mw\lt((\j_{k'_1,j'_2}A)\cap\X_{\ov\g}(x)\rt)\,d\h^k(x) =\int_{\X_\g} \M(A\cap\X_{\ov\g}(x))\,d\h^k(x).
\]
Summing over  $\g$, we obtain
\[
\Mw_{\Sl}(\j_{k'_1,k'_2}A)=\sum_{\substack{\g\in I^n_k\\(|\g^1|,|\g^2|)=(k'_1,k'_2)}} \int_{\X_\g} \M(A\cap\X_{\ov\g}(x) )\,d\h^k(x).
\]
Summing again over $(k'_1,k'_2)$ leads to
\be\label{proof_thm_MSlj=MSl_1}
\sum_{(k'_1,k'_2)\in D_k} \Mw_{\Sl}(\j_{k'_1,k'_2}A)=\M_{\Sl}(A),
\ee
which is point~(i) of the theorem. Now, for $A\in\FF^G_{k,(k_1,k_2)}(\R^n)$ we have by definition $\j_{k'_1,k'_2}A=0$ for $(k'_1,k'_2)\ne(k_1,k_2)$ and the above identity reduces to
\[
\Mw_{\Sl}(\j_{k_1,k_2}A)=\M_{\Sl}(A),
\]
which is the first identity of point~(ii) of the theorem. Eventually, by~\cite[Proposition~4.9]{GM_tfc} for $A\in\FF^G_{k,(k_1,k_2)}(\R^n)$, 
\[
\j_{k'_1,k'_2}\pt A=\begin{cases}
\pt_1\j_{k_1,k_2}A&\text{if }(k'_1,k'_2)=(k_1-1,k_2),\\
\pt_2\j_{k_1,k_2}A&\text{if }(k'_1,k'_2)=(k_1,k_2-1),\\
\qquad 0&\text{in the other cases}.
\end{cases}
\] 
Applying~\eqref{proof_thm_MSlj=MSl_1} with $\pt A$ in place of $A$, these identities yield
\[
\Mw_{\Sl}(\pt_1\j_{k_1,k_2}A)+\Mw_{\Sl}(\pt_2\j_{k_1,k_2}A)=\M_{\Sl}(\pt A),
\]
which is the second identity of~(ii). This ends the proof of the theorem.
\end{proof}

\subsection{A counterexample (proof of Proposition~\ref{prop_j_not_onto})}
\label{Ss_ctrex}

We assume $n_1=n_2=1$ and that $G$ has a nonzero element $g$. We build below a rectifiable $(0,1)$-chain which is not in the image of $\FF^G_{1,(0,1)}(\R^2)$ by $\j_{0,1}$ and a $(0,0)$-chain in $\FF^G_{0,0}(\R^2)$ which is not in the image of $\FF^G_{0,(0,0)}(\R^2)$ by $\j_{0,0}$. This is sufficient for proving Proposition~\ref{prop_j_not_onto} as we obtain counterexamples in higher dimension by tensorizing (and possibly exchanging the roles of $\X_\a$ and $\X_\ova$).  

\begin{proof}[Proof of Propositions~\ref{prop_j_not_onto}]
We consider the $2$-chain $Q'=gT$ where $T$ is the (positively oriented) triangle,
\[
\{(x^1,x^2)\in\R^2:x^1,x^2\ge0, x^1+x^2\le1\}.
\]
Recall that $Q'\in \FF_2^G(\R^2)=\FF_{1,1}^G(\R^2)$. Let us set  $A':=\pt_1Q'\in\FF^G_{1,0}(\R^2)$ and $B':=\pt_1\pt_2Q'=-\pt_2\pt_1Q'\in\FF^G_{0,0}(\R^2)$. We show below that $A'$ and $B'$ are counterexamples for proving~(i) and~(ii) respectively. \medskip

\noindent
\textit{Step~1, proof of (i).} 
Let us check that $A'$ has finite mass. Denoting, 
\[
L:=\lt\{(x^1,x^2)\in\R^2:x^1+x^2=1\rt\},
\]
we observe that $A'=\pt_1Q'$ is supported  in $L\cup\{0\}\t[0,1]$. Indeed, denoting for $j\ge1$ and $i\ge 0$, $x_{j,i}:=(1-t_{j,i},t_{j,i})$ with $t_{j,i}:=(i+1/2)2^{-j}$, $A'$ is the limit as $j\up\oo$ of the sequence of polyhedral $(0,1)$-chains
\[
R'_j:=-g\lb (0,e_2)\rb + \sum_{i=0}^{2^j-1}g \lb(x_{j,i}-2^{-j-1}e_2,x_{j,i}+2^{-j-1}e_2)\rb.
\]
Since
\[
\supp R'_j\sub\lt(L\cup\{0\}\t[0,1]\rt)+\lt(\{0\}\t[-2^{-j-1},2^{-j-1}]\rt),
\]
we get that $A'$ is supported  in $L\cup\{0\}\t[0,1]$.\\
We compute, using the commutativity of $\pt_1$ with slicing along coordinate lines,  
\begin{equation}\label{sliceexample}
\Mw_{\Sl}(A')=\int_\R \Mw\lt(A'\cap(te_2+\R e_1)\rt)dt=\int_0^1\Mw\lt(g\lb(1-t,t)\rb-g\lb (0,t)\rb\rt)ds=2|g|_G<\oo.
\end{equation}
Hence $A'$ has finite mass and since it is supported on the $1$-rectifiable set $L\cup\{0\}\t[0,1]$, the tensor chain $A'$ is rectifiable.

Let us denote $A'_0:=A'\restr \{0\}\t\R$ and $A'_1:=A'\restr L$. We have $A'=A'_0+A'_1$ and since 
\[
A'_0=-\lb(0,e_2)\rb=\j_{0,1}(-\lb(0,e_2)\rb),
\]
with $-\lb(0,e_2)\rb\in \PP_{0,1}(\R^2)\sub\FF^G_{1,(0,1)}(\R^2)$, it is enough to prove that $A'_1\not\in\j_{0,1}(\FF^G_{1,(0,1)}(\R^2))$. 

Let us assume by contradiction that there exists $A_1\in\FF^G_{1,(0,1)}(\R^2)$ such that $\j_{0,1}A_1=A'_1$. Since $\j_{1,0}A_1=0$, by Theorem~\ref{thm_MSlj=MSl}, we have 
\be\label{A1not0}
\M_{\Sl}(A_1)=\Mw_{\Sl}(A'_1)=\Mw_{\Sl}(A')-\Mw_{\Sl}(A'_0)=|g|_G\ \in(0,\oo).
\ee
Moreover, since $\j$ is one-to-one and commutes with restrictions, there holds $A_1=A_1\restr L$ (in particular $A_1$ is rectifiable). Denoting by $\pi$ the orthogonal projection on $\R e_1$, we set $\wh A_1:=\pi\pf A_1$.\\
Now, we claim that, 
\[
\j_{0,0}\pi\pf A=\pi\pf\j_{1,0}A\qquad\text{for every }A\in\FF^G_1(\R^2). 
\]
Indeed, the relation holds true for polyhedral tensor $1$-chains and extends by continuity. Here, since $A_1\in\FF^G_{1,0}(\R^2)$, there holds $\j_{0,0}\wh A_1=\j_{0,0}\pi\pf A_1=\pi\pf\j_{0,1}A_1=0$ and we get, 
\be\label{proof_j_not_onto_1}
\wh A_1=0.
\ee
Eventually, we define a Lipschitz continuous mapping $f:\R^2\to\R^2$ by $f(x^1,x^2):=(x^1,1-x^1)$. There holds $f\circ\pi=\Id$ on $L$ and since $A_1$ is supported in $L$, we have
\[
A_1=\Id\pf(A_1\restr L)=f\pf\pi\pf A_1=f\pf \wh A_1\st{\eqref{proof_j_not_onto_1}}=0.
\]
This contradicts~\eqref{A1not0} and shows that $A'\not\in\j_{0,1}(\FF^G_{1,{0,1}}(\R^2))$.\medskip

\noindent
\textit{Step~2, proof of (ii).} 
Let us now show that $B'$ provides a counterexample for~(ii). We first check that $B'$ does not vanish. Assuming by contradiction that $B'=0$, we compute
\[
 \pt\lt[\i(\pt_2 Q')\rt]=\i(\pt_1\pt_2Q')=\i B'=0.
\] 
Hence $\i(\pt_2 Q')$ is a $1$-cycle over $\R e_1$, so it identifies with a constant $L^1$-function on $\R e_1$ and we get $\pt_2 Q'=0$. Using the operator $\ov \i$, symmetric to $\i$, obtained by swapping the roles of $\X_\a=\R e_1$ and $\X_{\ova}=\R e_2$, we deduce that $\pt  \ov \i Q'= \ov \i \pt_2 Q'=0$ and thus $\ov \i Q'=0$. This would lead to $Q'=0$   which is in  contradiction with $\M(Q')=|g|_G/2\ne0$. We conclude that 
\be\label{B'not0}
B'\ne0.
\ee

Next, let us establish that $B'\not\in\j_{0,0}(\FF^G_0(\R^2))$. For this we assume by contradiction that there is some $B\in \FF_0^G(\R^n)$ with $B'=\j_{0,0} B$. As in the previous step we observe that $A'=\pt_1Q'$ and $A''=\pt_2Q''$ are supported respectively in $L\cup\{0\}\t[0,1]$ and $L\cup[0,1]\t\{0\}$. Since $B'=\pt_1A''=-\pt_2A'$, we deduce that $B'$ is supported in
\[
\supp(\pt_1Q') \cap\supp(\pt_2Q')\sub L.
\]
By Theorem~\ref{thm_supports}, we get that $B$ is also supported in $L$. In particular, for almost every $s,t\in\R$, the slices $B\cap [se_1+\R e_2]$ and $B\cap [te_2+\R e_1]$ are of the form $g'\lb x\rb$ for some $x\in L$. More precisely, using the operator $\chi$, we have for almost every $s,t\in\R$,
\[
B\cap(se_1+\R e_2)=\chi\lt(B\cap(se_1+\R e_2)\rt) \lb (s,1-s) \rb, \quad B\cap(te_2+\R e_1)=\chi\lt(B\cap(te_2+\R e_1)\rt) \lb (1-t,t) \rb.
\]
We compute (for almost every $s\in\R$),
\begin{align*}
\chi\lt(B\cap(se_1+\R e_2)\rt)
&=\chi\lt((\j_{0,0}B')\cap(se_1+\R e_2)\rt)\\
&=\chi\lt(\j_{0,0}\lt[B'\cap(se_1+\R e_2)\rt]\rt)   \\
&=\chiw\lt(B'\cap(se_1+\R e_2)\rt)\\
&=\chiw\lt((\pt_1\pt_2Q')\cap(se_1+\R e_2)\rt)\\
&=\chiw\lt(\pt_1\lt[(\pt_2Q')\cap(se_1+\R e_2)\rt]\rt),
\end{align*}
where we used~\eqref{form_chiw_2} for the third identity.
Now, $C'(s):=(\pt_2Q')\cap(se_1+\R e_2)$ is a chain in the vertical line $se_1+\R e_2$ so $\pt_1C'(s)=0$. It follows that $B\cap(se_1+\R e_2)$ vanishes for almost every $s$. Symmetrically $B\cap(te_2+\R e_1)=0$ for almost every $t$ and we conclude that $\Mw_{\Sl}(B)=0$ and then that $B'=\j_{0,0}B=0$ which contradicts~\eqref{B'not0}. Hence $B'$ is not in the image of $\j_{0,0}$.
\end{proof}
\subsection*{Acknowledments}
M. Goldman is partially supported by the ANR SHAPO. B. Merlet is partially supported by the INRIA team RAPSODI and the Labex CEMPI (ANR-11-LABX-0007-01).

\bibliographystyle{alpha}
\bibliography{BibStripes}
\end{document}